\documentclass[12pt]{article}

\usepackage{amssymb}
\usepackage{amsmath,amsthm}
\usepackage{latexsym}
\usepackage{amsbsy}
\usepackage{amscd}
\usepackage{amsfonts}
\usepackage[usenames]{color}

\newtheorem{teorema}{Theorem}[section]
\newtheorem{proposicion}[teorema]{Proposition}
\newtheorem{lema}[teorema]{Lemma}

\newtheorem{nota}[teorema]{Remark}
\newtheorem{hecho}[teorema]{Fact}
\newtheorem{ejemplo}[teorema]{Example}

\def\P{\mathrm{P}}
\newcommand{\proju}{\P^}

\def\E{\mbox{Emb}\hspace{.2mm}}
\def\immto{\looparrowright}
\def\Spin{\mathrm{Spin}}
\def\Sq{\mathrm{Sq}}

\newcommand{\sq}{\mathrm{Sq}^}
\newcommand{\com}{\mathbb{C}}

\newcommand{\cua}{\mathbb{H}}
\newcommand{\bz}{\mathbb{Z}}

\newcommand{\lra}{\longrightarrow}

\title{On the embedding dimension of 2-torsion \\ lens spaces}
\author{Jes\'us Gonz\'alez, Peter Landweber, and Thomas Shimkus}
\date{\empty}

\begin{document}
\maketitle

\begin{abstract}
Using the $ku$- and $BP$-theoretic versions of Astey's cobordism 
obstruction for the existence of smooth Euclidean embeddings 
of stably almost complex manifolds, we prove that, for $e$ greater 
than or equal to $\alpha(n)$---the number of ones in the dyadic 
expansion of $n$---, the ($2n+1$)-dimensional $2^e$-torsion lens 
space cannot be embedded in Euclidean space of dimension $4n-2\alpha(n)+1$. 
A slightly restricted version of this fact holds for $e<\alpha(n)$. 
We also give an inductive construction
of Euclidean embeddings for $2^e$-torsion lens spaces. 
Some of our best embeddings are within one dimension of being optimal.
\end{abstract}

\noindent {\it 2000 MSC:} 57R40, 19L41, 55S45.

\noindent {\it Keywords and phrases:} Euclidean embeddings of lens spaces, 
connective complex $K$-theory, Brown-Peterson theory, Euler class, modified 
Postnikov towers.

\section{Main results}\label{start}
For a positive integer $m$
let $\nu(m)$ and $\alpha(m)$ denote, respectively, 
the exponent in the highest power of $2$ dividing $m$, and
the number of ones in the dyadic expansion of $m$.
Let $L^{2m+1}(2^e)$ stand for the ($2m+1$)-dimensional $2^e$-torsion lens 
space, the quotient of $S^{2m+1}$ by the standard (diagonal) 
action of $\mathbb{Z}/2^e$ (viewed as a subgroup of the unit 
circle $S^1$). 

\smallskip
Unless explicitly noted otherwise, all (non)embedding results are 
to be understood in the smooth sense. Yet, except for a few 
low-dimensional cases (carefully pinpointed in the text), 
our results are within
Haefliger's metastable range $2m\ge 3(n+1)$ where, for a smooth closed 
manifold $M^n$, the existence of a topological embedding 
$M^n\subset\mathbb{R}^m$ is equivalent to the existence of a smooth embedding
$M^n\subset\mathbb{R}^m$.

\smallskip
We study the (Euclidean) embedding dimension of $L^{2m+1}(2^e)$, 
that is, the dimension of the smallest Euclidean 
space where this manifold can be embedded.
We get no new (non)embeddings for $e=1$ 
(although we reconstruct some of the known optimal ones), 
as our methods are generalizations of ideas already used for real 
projective spaces. 
Instead, we seek to understand the role the exponent $e$ plays 
in determining the Euclidean embedding dimension of $L^{2m+1}(2^e)$.

\smallskip
Our first result is the following 
analogue of Astey's nonembedding theorem for real projective spaces. 

\begin{teorema}\label{main} Let $n$ and $e$ be positive integers with 
\begin{equation}\label{cota}
e\ge\min\{\alpha(n)-6,\alpha(n)+1-2^{\nu(n)}\},
\end{equation}
and set $\delta=\max\{0,\alpha(n)-e\}$. 
Then the lens space  $L^{2(n+\delta)+1}(2^e)$ 
does not admit an embedding in $\mathbb{R}^{4n-2\alpha(n)+1}$.
\end{teorema}

The more pleasant  situation holds for $\delta=0$ 
(what we call a {\it high-torsion} lens space), where~(\ref{cota})
holds for free, with Theorem~\ref{main} affirming the 
impossibility of embedding $L^{2n+1}(2^e)$ in $\mathbb{R}^{4n-2\alpha(n)+1}$.
For $\delta>0$ (a {\it low-torsion} lens space), the conclusion is
similar, but besides the extra\footnote{Motivated
by the situation for $e=1$ in~\cite{Asteyembedding}, 
it was conjectured in~\cite{cfimm} that the hypothesis~(\ref{cota}) 
can be removed in the low-torsion case.} hypothesis~(\ref{cota}), 
one needs to add the little $(2\delta)$-correcting term 
to the dimension of the lens space.

\smallskip
Theorem~\ref{main} is better (and usually much stronger) 
than previous nonembedding results with $e\geq2$. 
For instance, the best nonembedding that~\cite{KS} gives for $L^{2n+1}(2^e)$ 
(using Atiyah's $\gamma$-operations) is in dimension
roughly $3n+1$. Notable exceptions are the (non)immersion results 
in~\cite{w.tom,tom} (but Theorem~\ref{main} applies for any $\alpha(n)$).

\smallskip
As for positive results, we start by recalling that 
the Haefliger-Hirsch-Massey-Peterson (HHMP) 
general embedding result (\cite{MP}, 
see also~\cite{VG} for the case of odd~$n$) gives  
an embedding $L^{2n+1}(2^e)\subset \mathbb{R}^{4n+1}$ for any pair $(n,e)$.
Theorem~\ref{enepot} below shows this is in fact optimal when $n$ is a 
power of 2 (improving, in such a case, Theorem~\ref{main} by one dimension).
Theorem~\ref{niam} deals with cases where $n$ is not a power of 2.

\begin{teorema}\label{enepot}
For any $e$, the HHMP-type embedding
$L^{2n+1}(2^e)\subset \mathbb{R}^{4n+1}$ is optimal when $\alpha(n)=1$.
\end{teorema}

\begin{teorema}\label{niam}
There are embeddings of the form
$L^{2n+1}(2^e)\subset\mathbb{R}^d$ for the triples 
$(e,n,d)$ indicated by the columns of Table~\ref{tabla1}, 
as well as for the special type of triples $(e,n,d)=(\leq2,7,26)$.
In Table~\ref{tabla1}, $\delta(1)=7$, $\delta(2)=9$, and 
$\delta(e)=10$ for $e\geq 3$.
\begin{table}[h]
\centerline{\footnotesize
\begin{tabular}{|c|l|l|l|l|}\hline
\rule{0mm}{4.5mm} $e$ & $\geq1$ & $\geq1$ & $\geq1$ & $\geq1$ \\ \hline
\rule{0mm}{4.5mm} $n$ & $2\ell+1$; $\;\;\ell\geq 1$ & $2\ell+1$; $\;\;$even 
$\ell$, $\alpha(\ell)\geq 2$ & $4\ell+3$, $\;\;\ell\geq 2$ & $4\ell+3$; 
$\;\;$even $\ell$, $\alpha(\ell)\geq 2$ \\ \hline
\rule{0mm}{4.5mm} $d$ & $8\ell+3$ & $8\ell+2$ & $16\ell+\delta(e)$ &
$16\ell+\delta(e)-1$\\\hline
\rule{0mm}{4.5mm} $\mathrm{\it eff}$ & $3$ & $4$ & $5$ & 
$6$ \\ \hline
\end{tabular}}
\caption{$L^{2n+1}(2^e)\subset\mathbb{R}^d$ (last row indicates 
embedding efficiency for $e=2$)}
\label{tabla1}
\end{table}
\end{teorema}

The only potentially non-smooth embeddings (outside Haefliger's range) 
in this result are those with $\ell=1$ 
in the second column of Table~\ref{tabla1}, that is, the embeddings $L^7(2^e)
\subset\mathbb{R}^{11}$. On the other hand, the case $\ell=0$ is indeed 
exceptional; as indicated in Remark~\ref{n12}, all 
3-dimensional lens spaces (smoothly) embed in $\mathbb{R}^5$, but no 
3-dimensional lens space embeds in $\mathbb{R}^4$.

\smallskip
The strength of Theorem~\ref{niam} is better appreciated 
by contrasting it with Theorem~\ref{main}. For instance, in case 
of high-torsion lens spaces, the embedding given in column $c$ of 
Table~\ref{tabla1} is within $d(c)$ dimensions of being optimal, where
\begin{equation}\label{efficiency}
d(c)=(2\alpha(\ell)-1,\,2\alpha(\ell)-2,\,2\alpha(\ell),\,2\alpha(\ell)-1)
\end{equation}
for $c=(2,3,4,5)$. These figures lead to a couple of interesting comments
about the potential optimality of our results.
Firstly, in terms of the efficiency $\mathrm{\it eff}$ of an 
embedding $M^m\subset\mathbb{R}^{2m-\mathrm{\it eff}}$, we observe that, 
although the embeddings in columns $c=4,5$ have better efficiency than 
those for $c=2,3$,~(\ref{efficiency}) indicates that the latter ones are, 
in general, closer to being optimal. This is of course interpreted
as saying that the difficulty in computing the embedding dimension of 
$L^{2n+1}(2^e)$ increases with $\alpha(n)$, a standard empirical fact for
real projective spaces. Secondly, note that
the only situation where $d(c)$ could be zero---thus indicating an optimal 
embedding---is for $c=3$ and $\alpha(\ell)=1$; unfortunately the last 
equality is ruled out for $c=3$. In Remark~\ref{utilidad} we discuss the
expectations for what the optimal embedding could turn out to be for this 
situation (see also Remark~\ref{eje7} and Example~\ref{ejegen}, the 
difficulty of the former being the first major obstacle in the field, 
and the main motivation for this paper). Now, at the opposite extreme of 
2-torsion (i.e., for real projective spaces), we point out (Remark~\ref{power})
that the cases with the lowest allowed value of $\alpha(\ell)$ in the fourth
column of Table~\ref{tabla1} are known to give optimal embeddings
for $e=1$. Likewise, the cases with the lowest allowed value of $\alpha(\ell)$
in the fifth column of Table~\ref{tabla1} give currently best known embeddings
for $e=1$.


\smallskip
We now comment on the methods (their origins and expectations) 
used in proving the three theorems above.

\smallskip
Following Astey's work~\cite{Asteyembedding} for real projective spaces, 
the proof of Theorem~\ref{main} extends, to the embedding realm, the 
nonimmersion results for lens spaces in~\cite{nonimm,cfimm}. 
In turn, these arose from Davis' strong nonimmersion 
result for real projective spaces~\cite{Davisstrongimmersion}
(see also~\cite{Asteyimmersion}). The form of all these results combines 
Euclidean dimension, manifold dimension,
and (for lens spaces) torsion in the fundamental group, 
in order to better reflect irregularities 
in the immersion and (now) embedding dimensions. 
For real projective spaces,\footnote{Technical problems prevented 
the first author from obtaining in~\cite{cfimm} the required general 
result that would have made these considerations work for lens spaces.} 
the strength of nonimmersion and nonembedding results of this type is 
due to the fact that, while involving a simple but general statement, 
they are often either the currently best known, or within a short 
distance of the best known. For instance,
as originally explained in~\cite{Davisstrongimmersion} and updated 
in~\cite{BDM},
whenever it is currently known that $\P^n$ does not immerse in $\mathbb{R}^m$,
then \cite{Davisstrongimmersion} affirms that $\P^{n+i}$ does not immerse in 
$\mathbb{R}^{m-j}$ for some nonnegative integers $i$ and $j$ with $i+j\leq 3$
(the last inequality can be improved to $i+j\leq 2$, 
if one excludes the nonimmersion results 
in~\cite{donzelov}). Such a comparison takes into account the recent 
works~\cite{BDM,DMtmf,KWa,KWb} obtained with some of the most sophisticated 
homotopy technology currently available. 
One should keep in mind, though, that for a 
{\it fixed} projective space, current nonimmersion results can 
improve the original~\cite{Davisstrongimmersion} by arbitrarily large 
Euclidean dimensions---but this is not new: the much older~\cite{James} 
has the same effect over~\cite{Davisstrongimmersion}. At any rate,
it is interesting to keep in mind the possibility that 
the recent advances in homotopy methods could turn out 
to be particularly helpful in settling some of the small gaps pinpointed 
in~(\ref{efficiency}).

\smallskip
Theorem~\ref{enepot} is not really new; as we will see in 
Section~\ref{classical}, it follows from the exact same argument
used in~\cite{Mah62} for the case $e=1$ (see also~\cite{Levine}). 

\smallskip
The proof of Theorem~\ref{niam} 
adapts, for lens spaces, the inductive Euclidean embedding
constructions done in~\cite{MM} for real projective spaces $\P^n$. 
We need to proceed with care though, so as to avoid the flaw in~\cite{MM}
coming from using the (not properly argued) immersions in~\cite{GMflaw}
(see Remark~\ref{otros}). One of the main ingredients in~\cite{MM} comes from 
Milgram's linear algebra techniques in~\cite{milgramlinear} for 
constructing, over a given space $\P^{n+k}$, spanning fields with the 
property that enough of them are trivial over a given smaller 
$\P^{n}$. But for 
lens spaces we need to replace such a linear algebra component by a careful 
obstruction theory analysis in the form of modified Postnikov towers. As
a side consequence we get the fact that Milgram's linear algebra 
input in the (start of the) 
inductive method in~\cite{MM} not only has the power to produce, when combined 
with the embeddings for $\P^7$ and $\P^{15}$ in~\cite{Rees}, 
many optimal embeddings of real projective spaces
(see Remark~\ref{power}), but can actually 
be extended, via obstruction theory, to the 
fatter\footnote{Section~\ref{comp} discusses the 
sense in which $2^e$-torsion lens spaces get fatter 
as $e$  increases.} lens spaces in order to trigger a reasonably 
strong (and in a sense optimal---see the final remarks in 
section~\ref{normal}) inductive construction of their embeddings. 
In this respect, it is to be noted that not all known methods for constructing 
embeddings of real projective spaces have a chance to work for higher 2-torsion
lens spaces. Indeed, some of the projective space arguments known to date
use, in an essential way, phenomena inherent to real projective spaces, thus 
constructing low-dimensional embeddings that, in general, will
not have to be true for high-torsion lens spaces. So, part of the problem 
consists of identifying methods that do construct (reasonably strong) 
embeddings for general $2^e$-torsion lens spaces.
This is indeed one of the goals set for this paper, with Proposition~\ref{fin}
its most notable accomplishment.

\smallskip
Observe that having concentrated in this paper on lens spaces 
of torsion a power 
of 2 is not a real restriction. For one, as we already mentioned, except for 
a few low-dimensional cases, all of our results are within Haefliger's 
metastable range. Therefore, using the canonical projection 
$L^{2m+1}(2^e)\to L^{2m+1}(2^ek)$ in~\cite[Theorem~B]{rigdon}, we see that, 
for odd $k$, our embedding results for $L^{2m+1}(2^e)$ automatically apply to 
$L^{2m+1}(2^ek)$.

\smallskip
The paper is organized as follows.
Theorem~\ref{main} is proved in Section~\ref{mainproof}
by a straightforward application 
of Astey's general $MU$-obstruction result~\cite{Asteyembedding} 
for Euclidean embeddings of almost complex manifolds, together with the
$ku$-calculations in~\cite{nonimm} and the $BP$-calculations in~\cite{cfimm}. 
Section~\ref{classical} 
is devoted to the proof of Theorem~\ref{enepot}, and to the description of 
some easy embedding results, mainly setting grounds for comparing with the 
embeddings in Section~\ref{newcomp}. The ideas in~\cite{MM} are worked out 
in Section~\ref{newcomp} from the point of view of lens 
spaces in order to prove Theorem~\ref{niam}. 
Section~\ref{normal} gives the obstruction theory
details that replace Milgram's linear algebra input.
The goal of the final Section~\ref{comp} is to pinpoint key subtleties 
arising when we compare the behavior of 
the (immersion and) embedding dimensions of $2^e$-torsion lens spaces, 
as $e$ varies.

\smallskip
This research was completed while the first author was on a 
sabbatical break visiting CIMAT, in Guanajuato, M\'exico. 
It is a pleasure to thank faculty and staff at CIMAT for creating 
the conditions 
that led to an enjoyable and productive visit. A CONACyT sabbatical 
scholarship provided the required financial support.

\section{Proof of Theorem~\ref{main}}\label{mainproof}
Let $M$ be a smooth compact stably almost complex 
manifold, with stable normal bundle $\nu$
of complex dimension $d$. Let $h^*$ denote a multiplicative complex-oriented 
cohomology theory whose ring of coefficients is concentrated in even dimensions,
and where 2 is not a zero divisor. Let 
$\xi_k$ stand for the canonical real line bundle over the $k$-dimensional 
real projective space $\P^k$.
Then, as proved in~\cite{Asteyembedding}, $M$ does not admit a Euclidean 
embedding with codimension $2\ell$
provided the following two conditions hold with $k=d-\ell$:

\smallskip
\begin{itemize}
\item[(a)] The $h^*$-Euler class of $\nu\otimes\xi_{2k}$ is nontrivial.

\smallskip
\item[(b)] The external product $h^{\mathrm{even}}(M)\otimes_{h_*} h^*(\P^{2k})
\rightarrow h^{\mathrm{even}}(M\times \P^{2k})$
is an isomorphism.
\end{itemize}

In the case $M=L^{2m+1}(2^e)$ we take $\nu=(2^N-m-1)\eta$, where $N$ is any
large positive integer,
and where $\eta$ is obtained as the pull-back, under the canonical 
projection $L^{2m+1}(2^e)\rightarrow \mathbb{C} \P^m$, of the complex Hopf line 
bundle over 
the $m$-dimensional complex projective space. 

\begin{proof}[Proof of Theorem~\ref{main}]
Under the above conditions, and in order to rule out a 
possible ($4n-2\alpha(n)+1$)-dimensional
Euclidean embedding of $L^{2m+1}(2^e)$, for $m=n+\delta$, 
the bundle to consider in~(a) above is 
\begin{equation}\label{bundle}
(2^N-m-1)\eta\otimes\xi_{2k}, 
\end{equation}
where $k=2^N-2n+\alpha(n)-1$. If we let 
$h^*=ku^*$, connective complex $K$-theory, for $\delta=0$, and $h^*=BP^*$,
Brown-Peterson theory at the prime 2, for $\delta>0$, then
the main results in~\cite{nonimm} (for $\delta=0$)
and in~\cite{cfimm} (for $\delta>0$) assert that the 
$h^*$-Euler class of~(\ref{bundle}) is nontrivial. The result 
then follows since, in our case ($M=L^{2m+1}(2^e)$), condition
(b) above is a restatement of Proposition~3.1 in~\cite{nonimm}.
\end{proof}

\begin{nota}{\em
In~\cite[Proposition 2.1]{cfimm} it is proved that 
if $L^{2m+1}(2^e)$  were to immerse in $\mathbb{R}^{4n-2\alpha(n)}$, 
then the bundle~(\ref{bundle})
would have a nowhere zero section. Thus, in retrospect, the 
$h^*$-calculations in~\cite{nonimm} and~\cite{cfimm}
simultaneously yield the non-immersibility 
and non-embeddability of $L^{2m+1}(2^e)$ 
($m=n+\delta$, as in the proof of
Theorem~\ref{main}) in Euclidean dimensions
$4n-2\alpha(n)$ and $4n-2\alpha(n)+1$, respectively.
(This is one of the reasons why the remarks about Theorem~\ref{main}
in Section~\ref{start} mainly concern immersions rather than embeddings.)
}\end{nota}

\begin{nota}\label{n23}{\em
Astey's nonembedding result is indirectly based on the 
triviality of the (generalized) Euler 
class to the normal bundle for an 
embedding in a Euclidean space (see~\cite[Corollary~11.4]{MS}). Here we 
point out firstly that, when the above observation
is used for singular cohomology with 
$\mathbb{Z}$ coefficients (rather than $ku$ or $BP$), 
we obtain the well known Fact~\ref{wellknown} below, and secondly that,
together with Fact~\ref{parallel} in Section~\ref{comp},
this yields the impossibility of finding a Euclidean
embedding with codimension $2$ for any $2^e$-torsion lens space 
of dimension $\ge 5$. The only notable {\it possible} exception is the
parallelizable $7$-dimensional real projective space $\P^7$ (see
Remark~\ref{eje7} 
at the end of the paper). In particular, we recover cases $n=2,3$ 
of Theorem~\ref{main} (for high-torsion lens spaces in the case of $n=3$). 
}\end{nota}

\begin{hecho}\label{wellknown}{\em
{\it The only compact smooth orientable 
manifolds $M$ that can be embedded in Euclidean space with 
codimension $2$ are necessarily stably parallelizable}
(as observed in~\cite{HS}, the converse is in general false).
Indeed, the normal 2-plane bundle $\nu_M$ of such
an embedding would be orientable with trivial Euler class.
But SO(2)$=S^1$, so that being the realification of a 
complex line bundle with vanishing first Chern class, $\nu_M$ is 
in fact trivial, and consequently $M$ is stably parallelizable.
}\end{hecho}

\begin{nota}\label{n12}{\em
As for low (i.e.,~1 or 2) codimension Euclidean embeddings of 
lens spaces, Remark~\ref{n23}
does not consider the case of lens spaces of dimension 1 or 3. 
Luckily, these are well understood. Of course 
$S^1$ embeds optimally in $\mathbb{R}^2$. On the other hand,
according to~\cite{Hantzsche} no $3$-dimensional
lens space embeds into $\mathbb{R}^4$; 
however they all (smoothly) embed in $\mathbb{R}^5$, 
as follows from~\cite{Hirsh3}.
}\end{nota}

\section{Proof of Theorem~\ref{enepot} (early methods)}\label{classical}
We have just indicated in Remark~\ref{n12} that
the case $n=1$ in Theorem~\ref{enepot} is well known.

\begin{proof}[Proof of Theorem~\ref{enepot}\, for $n>1$] 
We derive a contradiction
from assuming an embedding $L^{2n+1}(2^e)\subset\mathbb{R}^{4n}$. 
The argument uses singular cohomology groups with mod 2 coefficients 
(which will be suppressed from the notation). 
The Gysin sequence of the 
normal sphere bundle $E$ (with projection $\pi$) for
the hypothesized embedding reduces to split short exact sequences
$$
0\to H^q(L^{2n+1}(2^e))\stackrel{\;\pi^*}{\longrightarrow} 
H^q(E)\stackrel{\psi}{\longrightarrow}  H^{q-2n+2}(L^{2n+1}(2^e))\to 0
$$
(the homomorphism after $\psi$ is multiplication by the Euler class,
which is trivial as mentioned in Remark~\ref{n23}).
For dimensions $q<4n-1$, the splitting is made geometrically explicit 
in~\cite[Section~3]{masseypacific} by exhibiting a subalgebra $A^*$ such that 
\begin{itemize}
\item[(a)] $H^q(E)=\pi^*(H^q(L^{2n+1}(2^e)))\oplus A^q$, for $0<q<4n-1$, 
\item[(b)] $A^{4n-1}=0$, and
\item[(c)] $A$ is stable under cohomology operations.
\end{itemize}
Consequently, in the dimensions of (a), $\psi$ restricts to an isomorphism
from $A^q$ onto $H^{q-2n+2}(L^{2n+1}(2^e))$.
In particular, there is a well defined element $a\in A^{2n-2}$ with
$\psi(a)=1$. Furthermore, as observed in~\cite[page 784]{masseypacific},
each element $\omega\in H^*(E)$ can be expressed uniquely in the form
\begin{equation}\label{unique}
\omega=\pi^*(u)+a\cdot \pi^*(v)
\end{equation}
where $u,v\in H^*(E)$. As a last piece of notation, 
let $x,y\in H^*(L^{2n+1}(2^e))$ stand for 
the nontrivial classes
in dimensions~1 and 2, respectively. They are 
connected by the relation $\beta_e(x)=y$, where 
$\beta_e$ is the Bockstein associated to the extension 
$0 \to \mathbb{Z}/2^e \to \mathbb{Z}/2^{2e} \to \mathbb{Z}/2^e \to 0$.
We recall $H^*(L^{2n+1}(2^e))$ is generated 
by $x$ and $y$  subject to the relations 
\begin{equation}\label{relations}
y^{n+1}=0 \mbox{ \ and \  } x^2=\varepsilon y 
\end{equation}
where $\varepsilon=1$ for $e=1$,  and $\varepsilon=0$ for $e>1$.

Start with the class $\omega\in A^{2n-1}$ corresponding to $x$ 
under $\psi$, which 
must clearly have the form $\omega =\delta \pi^*(xy^{n-1})+a\cdot\pi^*(x)$, for 
$\delta\in\mathbb{Z}/2$. Property (c) above implies that
$$
\delta \pi^*(y^n)+a\cdot\pi^*(y)=\beta_e\left(\delta \pi^*(xy^{n-1})+
a\cdot\pi^*(x)
\right)
$$
(as in~\cite{Mah62}, one uses here the fact that $a$ is 
the mod 2 reduction of a similar 
integral class) lies in $A$, whereas property (b) above yields 
\begin{equation}\label{product}
a^2\cdot \pi^*(xy)=\left(\rule{0mm}{3.6mm}
\delta \pi^*(y^n)+a\cdot\pi^*(y)\right)\left(\rule{0mm}{3.6mm}
\delta \pi^*(xy^{n-1})+
a\cdot\pi^*(x)\right)=0
\end{equation}
(the first equality uses the relation $y^{n+1}=0$, recalling $n>1$).
On the other hand,~\cite[lemma on page~785]{masseypacific} claims that
the expression~(\ref{unique}) for 
$a^2=\Sq^{2n-2}(a)$ has the form $a^2=\pi^*(a')+a\cdot 
\pi^*(\overline{\mbox{W}}_{2n-2})$, where $\overline{\mbox{W}}_{2n-2}$ 
is the normal Stiefel-Whitney class of $L^{2n+1}(2^e)$. But 
a straightforward computation using the 
hypothesis $\alpha(n)=1$ gives $\overline{\mbox{W}}_{2n-2}=y^{n-1}$,
so that~(\ref{product}) reduces, by dimensional reasons (once again 
recall $n>1$), to
$a\cdot \pi^*(xy^n)=0$ or, by the uniqueness
of~(\ref{unique}), $xy^n=0$, in contradiction to~(\ref{relations}).
\end{proof}

We close this section with a rather geometric 
argument leading to an easy embedding result for lens spaces, which
partially generalizes Mahowald's embedding $\P^k\subset\mathbb{R}^{2k-2}$
(valid for $k>3$, with $k\equiv 3\mod{4}$, see~\cite[Theorem~7.2.2]{mah64}).
The (improved) full generalization is 
given by Theorem~\ref{niam}. 

\begin{proposicion}\label{spin}
$L^{2n+1}(2^e)$ embeds in $\mathbb{R}^{4n}$ provided $n\equiv 3\mod{4}$.
\end{proposicion}
\begin{proof} 
This is a consequence of Lemma~\ref{spingen} below and 
the fact proved in~\cite{thomas} that,
for $m\equiv 5,6,7\mod{8}$ and 
$m\geq 7$, every spin manifold $M^m$ embeds in 
$\mathbb{R}^{2m-2}$.
\end{proof}

\begin{lema}\label{spingen}
$L^{2n+1}(2^e)$ has a spin structure precisely for $n=0$, and 
for odd $n\ge 1$.
\end{lema}
\begin{proof} 
Since $L^{2n+1}(2^e)$ is orientable, everything reduces to a 
straightforward calculation of the second Stiefel-Whitney class
$w_2=w_2(L^{2n+1}(2^e))$.
We present this using Wu's formula $w=\Sq(v)$ (\cite[Theorem~11.14]{MS}), thus
avoiding the need for explicit descriptions of tangent bundles. 
Wu's first relation gives $v_1=w_1$, which we already observed to be trivial.
Then Wu's second relation reduces to $v_2=w_2$. Thus, it suffices to 
observe that $\Sq^2$ acts trivially on 
$H^{2n-1}(L^{2n+1}(2^e);\mathbb{Z}/2)$ precisely when $n$ is odd. 
(The mod $2$ cohomology ring for $L^{2n+1}(2^e)$ can easily be obtained 
from~\cite[Example~3E2]{Hatcher}, whereas the action of the Steenrod algebra 
follows from the well known situation for $\mathbb{C} \P^n$.)
\end{proof}

\section{Inductive construction of  embeddings}\label{newcomp}
Following the inductive methods in~\cite{MM}, we now
produce explicit Euclidean embeddings for $2^e$-torsion lens spaces.
In this section we use the shorter notation $L_{n,e}$ for $L^{2n+1}(2^e)$.

\smallskip
For $e\ge 1$, consider $\mathbb{Z}/2^e$ as the subgroup of $S^1$
of $2^{e\hspace{.05mm}}$th roots of unity. 
For $k,j\ge 0$, think of $S^{2(k+j+1)+1}$ as the join of $S^{2k+1}$ and 
$S^{2j+1}$ via the explicit homeomorphism $\phi\colon S^{2k+1}\star
S^{2j+1}\rightarrow S^{2(k+j+1)+1}$ sending $(x,y,t)$ into $(\sqrt{1-t}x,
\sqrt{t}y)$. Under this identification, the standard
$\mathbb{Z}/2^e$-action on $S^{2(k+j+1)+1}$ takes the form 
\begin{equation}\label{action0}
\omega(x,y,t)=(\omega x,\omega y,t).
\end{equation}
By restriction, this yields actions on the 
subsets $L',R'\subset S^{2k+1}\star S^{2j+1}$ determined by the 
conditions $t\le 1/2$ and $t\ge 1/2$, respectively. At the orbit space level 
we get a decomposition 
\begin{equation}\label{union}
L_{k+j+1,e}=L\,\cup\, R
\end{equation}
where $L$ (resp., $R$)
is the normal real disc bundle for the embedding of $L_{k,e}$ 
(resp., $L_{j,e}$) in $L_{k+j+1,e}$ coming from the first 
(resp., last) coordinates. Explicit models for $L$ and $R$ are then given by
\begin{equation}\label{parts}
L=\left.\left(S^{2k+1}\times D^{2j+2}\right)\right/\mathbb{Z}/2^e
\quad\mbox{and}\quad
R=\left.\left(D^{2k+2}\times S^{2j+1}\right)\right/\mathbb{Z}/2^e
\end{equation}
where, as indicated by~(\ref{action0}), the $\mathbb{Z}/2^e$ action is 
diagonal in both cases. More familiar descriptions are obtained 
from the following considerations.

\smallskip
Let $H_m$ stand for the canonical complex line bundle 
over $\mathbb{C} \P^m$, and let 
$\eta_m$ denote the complex conjugate bundle $H_m^*$. 
The map $\lambda \colon S^{2m+1}\times \mathbb{C}\rightarrow 
\mathbb{C} \P^m\times 
\mathbb{C}^{m+1}$ sending $(x,z)$ into $\left([x],zx\right)$, where
$[x]$ stands for the complex line  determined by $x$, satisfies 
$\lambda(x,\omega z)=\lambda(\omega x,z)$ for $\omega\in S^1$,
and produces the standard model
for $H_m$ as $S^{2m+1}\times_{S^1}\mathbb{C}$ ---the orbit space of the 
$S^1$ action on
$S^{2m+1}\times\mathbb{C}$ given by $\omega(x,z)=(\omega^{-1} x,\omega 
z)$. In particular, a model for $\eta_m$ is given by 
$(S^{2m+1}\times\mathbb{C})/S^1$ 
(diagonal action now). As a result, letting $\eta_{m,e}$ denote the 
pull-back of $\eta_m$ under the canonical map 
$L_{m,e}\rightarrow \mathbb{C} \P^m$,
we get that a model for the Whitney multiple $n\eta_{m,e}$ is given by
$(S^{2m+1}\times\mathbb{C}^n)/\mathbb{Z}/2^e$ (diagonal action, of course). 
Consequently, we have proved:

\begin{lema}\label{int1}
$L=D((j+1)\eta_{k,e})$ and $R=D((k+1)\eta_{j,e})$, the total spaces of 
the disc bundles for (the realifications of) 
$(j+1)\eta_{k,e}$ and $(k+1)\eta_{j,e}$. 
\end{lema}

The easy geometric argument in the proof of~\cite[Lemma~3.1]{steer} can 
now be adapted to the current situation in order to identify
$L\,\cap\, R=\left(S^{2k+1}\times S^{2j+1}\right)\left/
\left(\rule{0mm}{3.5mm}\mathbb{Z}/2^e\right)\right.$ 
(corresponding to $t=1/2$ under $\phi$) with the sphere 
bundle of the realification of the (exterior) tensor product
$\eta^*_{k,e}\otimes\eta_{j,e}$. But in order to directly apply the results 
in~\cite{MM}, we switch to their sphere bundle and mapping cylinder notation:
$L\,\cap\, R$ is the total space of both sphere bundles
$S((j+1)\eta_{k,e})$ and $S((k+1)\eta_{j,e})$. Moreover, letting
$\pi_k\colon S((j+1)\eta_{k,e})\to L_{k,e}$ and 
$\pi_j\colon S((k+1)\eta_{j,e})\to L_{j,e}$ stand for the bundle 
projections, the considerations in~(\ref{union}) and~(\ref{parts}) 
yield the following reinterpretation of Lemma~\ref{int1}.

\begin{lema}\label{int2}
$L$ and $R$ are, respectively, 
the mapping cylinders of $\pi_k$ and $\pi_j$. Likewise, $L_{k+j+1,e}$
is the double mapping cylinder $M(\pi_k,\pi_j)$.
\end{lema}

Here the notation is as in~\cite[Section~5]{MM}. In particular, 
Theorems~2.2 and~5.2 in~\cite{MM} yield the following inductive 
method for constructing embeddings of lens spaces (compare 
to~\cite[Proposition~1]{naka}).

\begin{proposicion}\label{inductive}
Let $L_{k,e}\subset \mathbb{R}^\alpha$,
$L_{j,e}\subset \mathbb{R}^\beta$, and 
$(k+1)\eta_{j,e}\subset \mathbb{R}^{\sigma+\beta}$ be embeddings\footnote{Only 
the first of the three hypothetical embeddings in Proposition~\ref{inductive} 
needs to be smooth.}$\!$, and assume that the first embedding admits 
$\sigma$ everywhere linearly independent normal sections. Then, there is a 
(topological) embedding $L_{k+j+1,e}\subset
\mathbb{R}^{\alpha+\beta+1}$, provided either one of 
the following two numerical conditions holds:
\begin{itemize}
\item[(i)] $\;\;\sigma+\beta>4j+2$.\vspace{-.7mm}
\item[(ii)] $\;\;\sigma+\beta=4j+2$ and $2k+3\le 8a+2^b$, where 
$\nu(2j+2)=4a+b$, with $0\le b\le 3$.
\end{itemize}
\end{proposicion}

\begin{nota}\label{diferencias}{\em
From~(\ref{union}) and~(\ref{parts}) we also get the explicit 
identifications $(j+1)\eta_{k,e}=L_{k+j+1,e}-L_{j,e}$ and 
$(k+1)\eta_{j,e}=L_{k+j+1,e}-L_{k,e}$. These are the models
used in~\cite{naka}.
}\end{nota}

For a fixed $k$, Proposition~\ref{inductive} allows us to start 
an inductive process (which we call a {\it round\hspace{.5mm}}) 
for constructing embeddings of lens spaces. The best results are 
obtained for rounds with $k$ one less than a power of
$2$ (when $\sigma_{k,e}$ in the next lemma is largest).
The first (third) embedding in the hypothesis of 
Proposition~\ref{inductive} is the ingredient triggering (feeding) 
each such round. The following result, a consequence of~\cite[Theorem~2.2]{MM} 
and the summary of immersions for lens spaces 
in~\cite[Section~1]{w.tom}, provides us with the triggering ingredient.

\begin{lema}\label{igniting}
There are embeddings $L_{k,e}\subset \mathbb{R}^{4k+2}$ with $\sigma_{k,e}$ 
everywhere linearly independent normal sections, for the triples 
$(k,e,\sigma_{k,e})$ indicated by the columns of Table~\ref{tabla2}.
\begin{table}[h]
\centerline{
\begin{tabular}{|c|c|c|c|c|}\hline
\rule{0mm}{4.5mm} $k$ & $3$ & $3$ & $3$ & $$1 \\ \hline
\rule{0mm}{4.5mm} $e$ & $1$ & $2$ & $\geq 3$ & $\geq 1$ \\ \hline
\rule{0mm}{4.5mm} $\sigma_{k,e}$ & $7$ & $5$ & $4$ & $3$ \\ \hline
\end{tabular}}
\caption{Values of $\sigma_{k,e}$}
\label{tabla2} 
\end{table}
\end{lema}

\begin{nota}\label{como}{\em
In getting the value $\sigma_{3,1}=7$, one needs to observe that 
the normal bundle to an immersion $P^7\subseteq\mathbb{R}^8$ is trivial---this 
is an easy calculation with Stiefel-Whitney classes.
Similarly, the value $\sigma_{1,e}=3$ requires the observation that the normal 
bundle to an embedding $L_{1,e}\subset\mathbb{R}^5$ is trivial---this 
uses the fact noted in the first statement of Remark~\ref{n23}.
}\end{nota}

\begin{nota}\label{otros}{\em
The thorough reader will note that the 
information recollected in~\cite{w.tom} allows us to 
go for one further round (with $k=7$
and $\sigma_{7,e}=7$ for $e\geq2$, or the $\sigma_{7,1}=8$ considered 
in~\cite{MM}). [By the way, the flaw in~\cite{MM} comes from their 
(improperly-argued) large values of
$\sigma_{k,1}$, for $k=2^\mu-1$ and $\mu>4$.] However, it is surprising 
to see how the rounds correponding to large $k$'s (with $\mu\geq3$) 
lose strength (even with the improperly-argued $\sigma_{k,1}$'s).
And in fact, the third round ($k=7$) produces no new information 
for $2^e$-torsion lens spaces having $e\geq2$. Remark~\ref{perdida} 
below gives 
further details focusing on the situation in~\cite{MM} for real projective 
spaces.
}\end{nota}

The ingredient feeding the two rounds we need (with $k=1,3$) comes from
Lemma~\ref{feeding} below. Its $e=1$ analogue is proved in~\cite{MM}
by a straightforward application of Milgram's linear algebra techniques
in~\cite{milgramlinear}. The general case will be proved in
the next section using obstruction theory. 

\begin{lema}\label{feeding}
For $\mu=1,2$ and positive integers $\ell$ and $e$,
set $\,i=2^\mu \ell-1$, and assume $\ell\geq2$ when $e>2=\mu$. Then 
$2^\mu\eta_{i,e}\subset\mathbb{R}^{4i+3}$. If in addition
$\ell$ is even with $\alpha(\ell)\geq 2$, then $2^\mu\eta_{i,e}\subset
\mathbb{R}^{4i+2}$.
\end{lema}

\begin{nota}\label{sharpen}{\em
While the first conclusion in this lemma is directly used within the 
two inductive rounds referred to above (to form columns 2 and 4 in 
Table~\ref{tabla1}), the second conclusion is used to get a one dimension
improvement, in certain cases, of the embeddings inductively obtained (to 
form columns 3 and 5 in Table~\ref{tabla1}).
}\end{nota}

We now give the straightforward deduction of Theorem~\ref{niam} from 
Proposition~\ref{inductive} and Lemmas~\ref{igniting} and~\ref{feeding}.

\begin{proof}[Proof of Theorem~\ref{niam}]
Let $(e,n,d)$ be as in the second column of Table~\ref{tabla1}, and
proceed by induction on $\ell\geq1$. For the start of the induction, the
embedding $L_{3,e}\subset\mathbb{R}^{11}$, apply
Proposition~\ref{inductive}(i) with $k=1$, $j=1$, $\alpha=5$, 
$\beta=5$ (coming from Remark~\ref{n12}), and $\sigma=2$ (as observed in 
Remark~\ref{como}), using Lemma~\ref{feeding} with $\mu=1$ and, of course,
$\ell=1$. Then, for the inductive step apply
Proposition~\ref{inductive}(i) with
$k=1$, $j=2\ell-1$, $\alpha=6$, $\beta=8\ell-4$ (one higher than 
the inductive hypothesis), and $\sigma=\sigma_{1,e}$ (as in 
Lemma~\ref{igniting}), using Lemma~\ref{feeding} with $\mu=1$ and, of course, 
the current inductive $\ell$. 

For $(e,n,d)$ as in the third column of Table~\ref{tabla1}, apply
Proposition~\ref{inductive}(ii), with $k=1$, $j=2\ell-1$, $\alpha=6$, 
$\beta=8\ell-5$ (coming from the first round above), 
and $\sigma=\sigma_{1,e}$ 
(as in Lemma~\ref{igniting}), using Lemma~\ref{feeding} with $\mu=1$ and 
$\ell=\frac{n-1}{2}$. 

Now let $(e,n,d)$ be
as in the fourth column of Table~\ref{tabla1} (we only consider 
the case $e\geq2$; see Remark~\ref{power} for a slight strengthening of the 
method in the original case $e=1$), and proceed by induction on $\ell\geq2$.
For the inductive step apply Proposition~\ref{inductive}(i) with
$k=3$, $j=4\ell-1$, $\alpha=14$, $\beta=16\ell+\delta(e)-15$ (one higher 
than the inductive hypothesis), and $\sigma=\sigma_{3,e}$ (as in 
Lemma~\ref{igniting}), using Lemma~\ref{feeding} with $\mu=2$ and, of course, 
the current inductive $\ell$. This time, in order to ground the induction 
($\ell=2$), we need to show the existence of an embedding
\begin{equation}\label{argue}
L_{7,e}\subset\mathbb{R}^{17+\delta(e)}.
\end{equation}
For $e\geq3$, this is given by the second column in Table~\ref{tabla1} 
(with $\ell=3$). But for $e=2$,~(\ref{argue}) is just the embedding for
the special type of triples in the statement of Theorem~\ref{niam}. 

In order to establish~(\ref{argue}) for $e=2$, apply 
Proposition~\ref{inductive}(i) with
$k=3$, $j=3$, $\alpha=14$, $\beta=11$ (coming from the start of the first
induction in this proof), and $\sigma=\sigma_{3,2}=5$ (as in 
Lemma~\ref{igniting}), using Lemma~\ref{feeding} with $\mu=2$ and 
$\ell=1$ ---still a valid case in Lemma~\ref{feeding} (all we need at 
this point is the weaker embedding $4\eta_{3,2}\subset\mathbb{R}^{16}$).

Finally, for $(e,n,d)$ as in the fifth column of Table~\ref{tabla1},
apply Proposition~\ref{inductive}(ii), with $k=3$, $j=4\ell-1$, $\alpha=14$, 
$\beta=16(\ell-1)+\delta(e)$ (coming from the 
second round above), and 
$\sigma=\sigma_{3,e}$ (as in Lemma~\ref{igniting}), using Lemma~\ref{feeding} 
with $\mu=2$ and $\ell=\frac{n-3}{4}$. 
\end{proof}

Since $\sigma_{1,e}=3$ and the embedding dimension of any $L_{1,e}$ is $5$,
the first inductive
round (and its improvements) produces embeddings (second and third columns in 
Table~\ref{tabla1}) whose Euclidean dimensions are independent of $e$. But the 
second round's output (fourth and fifth columns in Table~\ref{tabla1}) do 
depend on $e$. The next remark describes the situation (sharpened with 
the information in~\cite{Rees}) for $e=1$.  

\begin{nota}\label{power}{\em
Rees' PL topological embedding $\P^7\subset\mathbb{R}^{10}$ in~\cite{Rees}
improves by one dimension the embedding coming from the start of the 
first inductive round. When this information is fed into the start of
the second round, there results an embedding $\P^{15}\subset\mathbb{R}^{25}$, 
a corresponding improvement of~(\ref{argue}) in one dimension (for $e=2$),
but still 2 dimensions weaker than Rees' PL embedding $P^{15}\subset
\mathbb{R}^{23}$ in~\cite{Rees}. But when the latter embedding is fed
into the next step of the second round, there results the 
(Haefliger smoothable) embedding
$P^{23}\subset\mathbb{R}^{39}$, an optimal result according 
to~\cite{dontablas}. Moreover, this situation even shows that Milgram's
embedding $4\eta_{7,1}\subset\mathbb{R}^{31}$ in Lemma~\ref{feeding} is 
optimal, so that the corresponding sharpening in column 5 of 
Table~\ref{tabla1} indeed fails to apply in this case. These two phenomena 
repeat consistently throughout the second round (for $e=1$) yielding
the embeddings
\begin{equation}\label{out1}
\P^{8j+7}\subset\mathbb{R}^{16j+7} \mbox{ \ \ and \ \ }
4\eta_{4j-1,1}\subset\mathbb{R}^{16j-1}, \mbox{ \ for \ }j\geq2
\end{equation}
and
\begin{equation}\label{out2}
\P^{8j+7}\subset\mathbb{R}^{16j+6}, \mbox{ \ for even $j$ not a power of 2.}
\end{equation}
Note that the first embedding in~(\ref{out1}) gives the $e=1$ case of the
fourth column in Table~\ref{tabla1}. This embedding and, therefore, the
second embedding in~(\ref{out1}) are optimal, according 
to~\cite{dontablas}, if $j$ is a power of 2, that is, when the improvement
referred in Remark~\ref{sharpen} actually fails to apply. Likewise,~(\ref{out2})
gives the $e=1$ case of the fifth column in Table~\ref{tabla1}. It is worth 
noticing that, according to~\cite{dontablas},
the embedding in~(\ref{out2}) is
currently best known when $j$ is even and 
$\alpha(j)=2$.
}\end{nota}

\begin{nota}\label{perdida}{\em
Inductive rounds corresponding to 
values $k=2^\mu-1$ with $\mu\geq3$ have a dramatical loss of strength.
In fact, in the case of real projective spaces, this problem 
(noticed in the paragraph previous to Theorem~1.8 in~\cite{MM}) led to a 
rather weak lower bound for the embedding efficiency of $\P^n$---roughly 
$2\log_2(\alpha(n))$, the main (but faulty)
theorem in~\cite{MM}. The reason for 
this diminished strength of methods comes from the fact that, 
for $i$ as in Lemma~\ref{feeding}, 
the feeding ingredient $2^\mu\eta_{i,e}\subset\mathbb{R}^{4i+3}$ is 
(in~\cite{MM}'s wording) ``rarely satisfied'' when 
$\mu\geq3$ (not to mention the embedding in $\mathbb{R}^{4i+2}$), and
thus needs to be replaced with a higher dimensional (therefore weaker) 
Euclidean embedding. We offer here a simple numerical analysis (which requires
familiarity with the notation in~\cite[Lemma~1.5]{MM}) of how the 
problem arises in the case of real projective spaces. The $\ell$-th step 
in the $\mu$-th inductive round has the form
$$
L_{2^{\mu}-1,1}\subset\mathbb{R}^\alpha \mbox{ \ with $\sigma$ normal sections, 
\ }L_{\ell \cdot 2^{\mu}-1,1}\subset\mathbb{R}^\beta, \;\;\;2^\mu
\eta_{\ell \cdot 2^{\mu}-1,1}\subset\mathbb{R}^{\sigma+\beta}\Rightarrow \cdots.
$$
In these conditions, and in order for Proposition~\ref{inductive} to have
its strongest conclusion, it is necessary that $\sigma+\beta$ be (perhaps 
one more than) twice the dimension of $L_{\ell \cdot 2^{\mu}-1,1}$. This proves 
to be the case under the condition ``$p\leq\alpha(n)+\kappa(p,n)-
\alpha(p+1)$'' of~\cite[Lemma~1.5]{MM}. In our terms, 
such a condition easily translates into
\begin{equation}\label{milcon}
2^{\mu+1}-1\leq\alpha(\ell)+\mu+\kappa(\mu),
\end{equation}
where $\kappa(\mu)=1$ for $\mu=1$, and $\kappa(\mu)=4$ for $\mu\geq2$. 
Although~(\ref{milcon}) 
always holds when $\mu\leq2$, it indeed rarely holds for $\mu\geq3$. 
}\end{nota}

\section{Proof of Lemma~\ref{feeding}}\label{normal}
We continue to use last section's notation $L_{m,e}$ 
for the lens space $L^{2m+1}(2^e)$. 
Recall that a model for $\tau_{m,e}$, the tangent bundle of $L_{m,e}$, is 
given by the quotient of the space of pairs $(x,y)\in S^{2m+1}\times 
\mathbb{C}^{m+1}$, where $x$ and $y$ are perpendicular, by the diagonal
action of $\mathbb{Z}/2^e$. Thus $(x,y)$ and $(\omega x,\omega y)$ are 
identified in $\tau_{m,e}$ for $\omega\in\mathbb{Z}/2^e$.

\begin{lema}\label{nmlbl}
For $m\geq n\geq 0$, $(m-n)\eta_{n,e}$ is the normal bundle for the 
embedding $L_{n,e}\subset L_{m,e}$ coming from the first coordinates.
\end{lema}
\begin{proof}
Recall from the previous section that $(m-n)\eta_{n,e}$ is the
quotient space of $S^{2n+1}\times \mathbb{C}^{m-n}$ by the diagonal action
of $\mathbb{Z}/2^e$. Then the map $((x,y),(x,z))\mapsto (x,(y,z))$ produces
a linear monomorphism $\tau_{n,e}\oplus(m-n)\eta_{n,e}\hookrightarrow\tau_{m,e}$
over $L_{n,e}\subset L_{m,e}$ that identifies $\tau_{n,e}\oplus(m-n)\eta_{n,e}$
with the restriction of $\tau_{m,e}$ to $L_{n,e}$.
\end{proof}

Let $F(t)$ be the Hurwicz-Radon function giving the maximal number of
everywhere linearly independent vector fields on $S^t$.
The following result gives the basis for our obstruction theory approach
to Milgram's linear algebra input in~\cite{MM}. 

\begin{proposicion}\label{encaje}
Assume that the restriction to $L_{n,e}$ of the
stable normal bundle $\nu\colon L_{m,e}\to BO$ is represented by a 
$d$-dimensional bundle classified by the map $\nu'$ in the homotopy 
commutative diagram
\begin{equation}\label{diagram}
\begin{picture}(50,40)(0,25)
\put(0,50){$L_{n,e}$}
\put(80,50){$BO(d)$}
\put(0,-5){$L_{m,e}$}
\put(85,-5){$BO$}
\put(33,55){\vector(1,0){40}}
\put(49,59){\scriptsize$\nu'$}
\put(33,0){\vector(1,0){46}}
\put(53,4){\scriptsize$\nu$}
\put(100,42){\vector(0,-1){30}}
\put(10,38){\vector(0,-1){26}}
\put(13.1,38){\oval(6,7)[t]}
\end{picture}
\end{equation}

\vspace{1cm}\noindent
If $2m+d\geq4n+1$ and, in case of equality, $2(m-n)\leq F(2n+1)$,
then there is an embedding $(m-n)\eta_{n,e}\subset\mathbb{R}^{2m+d+1}$.
\end{proposicion}
\begin{proof}
Think of $\nu$ as an $M$-dimensional bundle for $M\gg0$. After
cancelling a few trivial sections, the $L_{n,e}$-restriction of 
the equality $\tau_{m,e}\oplus\nu=2m+1+M$ becomes 
$\tau_{n,e}\oplus\nu''=2m+d+1$, where $\nu''=(m-n)\eta_{n,e}\oplus\nu'$.
Pick an immersion $L_{n,e}\immto\mathbb{R}^{2m+d+1}$ with normal bundle
$\nu''$. If $2m+d>4n+1$, this immersion is regularly homotopic to
an embedding, an open tubular neighborhood of which can be identified
with $\nu''$ (producing, in particular, the required embedding for
$(m-n)\eta_{n,e}$). The case $2m+d=4n+1$ is treated in a similar way, except
that the regular deformation argument is replaced by~\cite[Theorem~2.2]{MM}.
\end{proof}

We next show how to reduce Lemma~\ref{feeding} to a
particular instance of Proposition~\ref{encaje}.
The reader will readily verify that, for $\mu$, $\ell$, and $e$ as 
in Lemma~\ref{feeding}, the numerical requirements in 
Proposition~\ref{encaje} are satisfied with $n=2^\mu\ell-1$, 
$m=2^\mu(\ell+1)-1$, and $d=2^{\mu+1}(\ell-1)-\lambda$, where 
\begin{equation}\label{lambda}
\lambda=\begin{cases}
1, & \mbox{when $\alpha(\ell)\ge2$ and $\ell\equiv 0 \mbox{ mod } 2$}; \\
0, & \mbox{otherwise}.
\end{cases}
\end{equation}
And, under these conditions, Lemma~\ref{feeding} becomes the 
conclusion of Proposition~\ref{encaje}. Therefore, settling the existence 
of a map $\nu'$ as in diagram~(\ref{diagram}) is the only missing 
task in order to complete the proof of Lemma~\ref{feeding}.
Moreover, since the stable normal bundle of $L_{m,e}$ is well-known to be
the $-(m+1)$ Whitney multiple  of the pull-back of the Hopf 
bundle $H_m$ (denoted simply by $H$ if no confusion arises) 
under the canonical projection $L_{m,e}\to\mathbb{C}\P^m$, 
our real goal becomes showing the existence of the 
dashed homotopy lifting in the following diagram. 
We stress that the hypothesis to this aim, namely 
$\mu$, $\ell$, and $e$ being as in Lemma~\ref{feeding}, with
$\lambda$ given by~(\ref{lambda}),
will be in force throughout the rest of the section.
\begin{equation}\label{lift}
\begin{picture}(50,30)(-49,35)
\put(0,50){$BO(2^{\mu+1}(\ell-1)-\lambda)$}
\put(55,0){$BO$}
\put(-93,0){$\mathbb{C}\P^{2^\mu\cdot\ell-1}$}
\put(-185,0){$L_{2^\mu\cdot\ell-1,e}$}
\put(-128,5){\vector(1,0){28}}
\put(-31,5){\vector(1,0){80}}
\put(-27,-7){\scriptsize $-2^\mu(\ell+1)H$}
\put(67,40){\vector(0,-1){23}}
\multiput(-147,15)(28,7){5}{\line(4,1){20}}
\put(-27,45.1){\vector(4,1){15}}
\end{picture}
\end{equation}

\vspace{.8cm}
\begin{nota}\label{aclaracion}\em 
The following considerations refer to the
case $\lambda=1$ in~(\ref{lambda}).
While the assumed parity of $\ell$ is an important
ingredient in the 
above derivation of Lemma~\ref{feeding} from Proposition~\ref{encaje},
no use is yet made 
of the condition $\alpha(\ell)\ge 2$. The latter 
requirement will shortly be identified as 
the relevant hypothesis for the corresponding (i.e., $\lambda=1$) 
construction of
the lifting in~(\ref{lift}). In fact, it will be convenient to 
construct a slightly stronger set of liftings (for
$\mu$, $\ell$, and $e$), namely, one where  
the restriction $\ell\equiv 0$ mod 2 is removed from the case 
$\lambda=1$ of~(\ref{lambda}).
\end{nota}

We next take care of a couple of easy cases of~(\ref{lift}).

\begin{proof}[{\bf Case} $\ell=1$] 
(\ref{lambda}) gives $\lambda=0$, so we need to show the
homotopy triviality of the horizontal composite in~(\ref{lift}). This is
standard for $e=1$, whereas the case $e\geq2$ follows from the 
$\widetilde{KO}$-calculations in~\cite{KS} (this is the point where we
require the hypothesis $e\leq2$ for $\mu=2$). Therefore we assume 
$\ell\geq 2$ from now on---which will allow us to get the lifting 
in~(\ref{lift}) even at the level of the complex projective space.
\end{proof}

\begin{proof}[{\bf Case} $\lambda=1$] 
(The cases with an even $\ell$ correspond to the `stronger'
liftings leading to the improvements referred to in Remark~\ref{sharpen}.)
For $\mu=1$, the existence of the required lifting was actually 
established in~\cite{naka2}
provided $\ell\equiv0$ mod 2---see the proof of 
Theorem~1 on pages 172--173 of that
paper. We  leave for the reader the verification that
the required lifting for $\mu=1$ (and any $\ell$)
follows from an argument
similar to the one we now describe for
the situation with $\mu=2$. The lifting problem we want to solve is

\begin{equation}\label{l1m2} 
\begin{picture}(50,30)(0,35)
\put(10,50){$BO(8(\ell-1)-1)$}
\put(55,0){$BO$}
\put(-85,0){$\mathbb{C}\P^{4\ell-1}$}
\put(-31,5){\vector(1,0){80}}
\put(-24,-7){\scriptsize $-4(\ell+1)H$}
\put(67,40){\vector(0,-1){23}}
\multiput(-55,20)(22,11){3}{\line(2,1){15}}
\put(-9,43){\vector(2,1){15}}
\end{picture}
\end{equation}

\vspace{1cm}\noindent
and, interpreted as an upper bound for the
geometric dimension of the horizontal map in~(\ref{l1m2}), 
its solution follows as a consequence of~\cite[Theorem~1.3(c)]{DM2}.
Indeed, in that result take 
\begin{equation}\label{otranot}
\epsilon=3, \;\;d=5, \;\;m=2\ell-3, \mbox{ \ and \ } \,p=2^N-4(\ell+1), 
\end{equation}
where $N\gg0$ (in fact, the summand $2^N$ would have to be 
replaced by any multiple of the order of the Hopf bundle over 
$\mathbb{C}\P^{4\ell-1}$, but this is immaterial for the $2$-primary
calculations below).
Then, with the notation as in~(\ref{otranot}),
the conditions implying the lifting in~(\ref{l1m2}) are

\begin{itemize}
\item $\nu\binom{p}{2m+2}\geq 1$;
\item $\nu\binom{p}{2m+4}\geq 3$;
\item $2m\geq d-\epsilon$.
\end{itemize}

\noindent
Table~1.9 in~\cite{DM2} imposes no further `additional conditions'. The third
inequality is immediate, while the first two are straightforward verifications
using the identities $\nu\binom{a}{b}=\alpha(b)+\alpha(a-b)-\alpha(a)$, 
$\;\alpha(a-1)=\alpha(a)-1+\nu(a)$, and $\alpha(2^N-a)=N-\alpha(a-1)$ for 
$N\gg 0$. 
For instance, the second inequality is verified as follows:
\begin{eqnarray*}
\lefteqn{\nu\binom{p}{2m+4} \;\; = \;\; \nu\binom{2^N-4(\ell+1)}{4\ell-2}} 
\\[1mm]
 & = & \alpha(4\ell-2) + \alpha(2^N-8\ell-2) - \alpha(2^N-4(\ell+1)) \\ 
 & = & \alpha(\ell-1) + 1 + N - \alpha(8\ell+1) - (N-\alpha(4\ell+3)) \\ 
 & = & \alpha(\ell-1) + 1 - \alpha(\ell) - 1 + \alpha(\ell) +2 \;\; = \;\; 
\alpha(\ell-1)+2 \\ 
 & = & \alpha(\ell) - 1 +\nu(\ell) +2 \;\; \geq \;\; 2-1+2 \;\; = \;\; 3,
\end{eqnarray*}
where the inequality in the previous line uses the 
hypothesis $\alpha(\ell) \ge 2$. 
In a similar way one checks that $\nu\binom{p}{2m+2}=\alpha(\ell)-1$.
\end{proof}

\begin{nota}\label{maslift}{\em
Stronger liftings than~(\ref{l1m2}) can be deduced from the Davis-Mahowald
technique~\cite{DM2} for certain odd values of $\ell$. 
For instance, when $\ell=2^a u +3$ with odd $u$ and $a\ge 2$, 
we get a lifting to $BO(8(\ell-1)-2)$ 
if $u=1$, and to $BO(8(\ell-1)-3)$ if $u>1$.
These follow from~\cite[Theorem~1.3(c)]{DM2} by replacing 
$\epsilon = 3$ in~(\ref{otranot}) with $\epsilon = 2$ 
and $\epsilon = 1$, respectively.
}\end{nota}

It remains to consider the lifting~(\ref{lift}) for  $\ell\geq 2$ and
$\lambda=0$. According to~(\ref{lambda})---as modified in 
Remark~\ref{aclaracion}---, this means in fact
\begin{equation}\label{tom}
\ell=2^\varrho, \mbox{ \ \ with \ } \varrho>0.
\end{equation}
For this, we first observe 
that the horizontal composite in~(\ref{lift}) factors through the 
corresponding quaternionic projective space as 
$$
L_{2^\mu\ell-1,e}\to \mathbb{C}\P^{2^\mu\ell-1}\to \mathbb{H}
\P^{2^{\mu-1}\ell-1}\to BO.
$$
(see~\cite[Lemma~5.4]{sanderson}).
Since the case with $\mu=2$ is far more complicated, we first 
dispose of the (rather elementary) situation for $\mu=1$. 
The point is that the fiber of $BO(4\ell-4)\to BO$ is 
($4\ell-5$)-connected so
that, at the level of the quaternionic projective space, the obstructions for 
the required lifting lie in trivial groups.

\medskip
We have saved the most interesting case for last, namely, 
the one having $\lambda=0$ and $\mu=2$, with $\ell$ as in~(\ref{tom}).
Thus, the proof of Lemma~\ref{feeding} will be complete
once we solve the instance of~(\ref{lift}) 
summarized by the following result (where we have set $j=\ell-1$,
an odd number in view of~(\ref{tom})).

\begin{proposicion}\label{fin}
The homotopy lifting problem

\begin{picture}(50,60)(-273,18)
\put(43,50){$BO(8j)$}
\put(55,0){$BO$}
\put(-55,0){$\mathbb{H}\P^{2j+1}$}
\put(-144,0){$\mathbb{C}\P^{4j+3}$}
\put(-90,5){\vector(1,0){28}}
\put(-1,5){\vector(1,0){50}}
\put(67,40){\vector(0,-1){23}}
\multiput(-100,20)(28,7){5}{\line(4,1){20}}
\put(20,50.1){\vector(4,1){15}}
\end{picture}

\vspace{1cm}\noindent
can be solved for any positive odd 
\footnote{We have just indicated stronger liftings for an even $j$.} 
integer $j$, where the horizontal composite classifies the 
realification of the Whitney multiple $-4(j+2)H_{4j+3}$. 
\end{proposicion}

We have already said that the instance of~(\ref{lift}) represented by
Proposition~\ref{fin} (as well as the previous instances of~(\ref{lift}) we 
have already solved) is proved 
in~\cite{MM} for real projective spaces by means of direct linear 
algebra constructions. The lack of such technique for 
higher torsion lens spaces led us to analyze alternative approaches
to Proposition~\ref{fin} that could not only explain in a 
simple form the 
corresponding $e=1$ instance in~(\ref{lift}), 
but that would allow us to obtain a suitable generalization to 
higher torsion lens spaces. The tool 
that proved to simultaneously solve these problems is Mahowald's theory of 
modified Postnikov towers (MPT's)~\cite{mah64}, as refined in~\cite{GM66}
---which is freely used from this point on.
Indeed, such an analysis is actually very simple in the case of real projective
spaces as the whole lifting can be sorted out through easy primary
indeterminacy arguments. This gave a first indication of the viability of the 
method in the general case. But its complete success comes from a careful
secondary indeterminacy analysis---the heart of the proof below---that
leads to the required lifting~(\ref{lift}) 
for every torsion lens space (in the form of Proposition~\ref{fin}).

\medskip
The remainder of this section is devoted to the construction of the 
dashed homotopy lifting in Proposition~\ref{fin}. 
The proof is a bit lengthy, so we divide it into three main steps.
The first two work for any (even or odd) $j$; 
the hypothesis of having an odd $j$ will be applied only near 
the end of the third step, in order to evaluate a certain nontrivial
secondary cohomology operation.

\smallskip 
Step 1. {\em Description of the MPT we use}.
Since $-4(j+2)H_{4j+3}$ has trivial first and second Stiefel-Whitney
classes, the classifying map of this bundle
can be further factored through $B\Spin$. 
We denote the maps in the resulting factorization as 
\begin{equation}\label{factorizacion}
\com\proju {4j+3}\stackrel{f}\rightarrow \cua \proju{2j+1}
\stackrel{g}\rightarrow B\Spin\stackrel{h}\rightarrow BO,
\end{equation}
and note that it suffices to lift the composite $gf$ to $B\Spin(8j)$. 
In the range under consideration the fiber of 
$B\Spin(8j)\rightarrow B\Spin$ is the stunted real projective space 
$\P^{8j+6}_{8j}$---the quotient of $\P^{8j+6}$ with $\P^{8j-1}$ collapsed
to a point. This is the Thom space of the 
bundle $8j\xi_{6}$ (notation as in Section~\ref{mainproof}), which is
a stable coreducible complex 
by standard $K$-theory considerations. Thus $\P^{8j+6}_{8j}$ splits 
as $S^{8j}\vee\P^{8j+6}_{8j+1}$,
with homotopy groups through dimension $8j+5$  given 
in the Adams chart (see \cite[Tables 8.1, 8.2, and 8.9]{meta})

\begin{picture}(100,128)(-45,-25)
\def\elt{\circle*{4}}
\def\mp{\multiput}
\put(100,0){\line(1,0){245}} 
\mp(120,0)(0,20){4}{\elt} 
\put(120,0){\vector(0,1){90}}
\mp(160,0)(0,20){2}{\elt} 
\mp(200,20)(0,20){2}{\elt} 
\mp(243,20)(0,20){3}{\elt} 
\mp(237,23)(0,20){2}{\elt} 
\mp(240,0)(0,20){1}{\elt} 
\mp(280,20)(0,20){1}{\elt} 
\put(114,-15){\scriptsize $8j$} 
\put(200,-2){\line(0,1){4}} 
\put(185,-15){\scriptsize $8j{+}2$} 
\put(265,-14){\scriptsize $8j{+}4$} 
\mp(280,-2)(40,0){2}{\line(0,1){4}} 
\put(243,20){\line(0,1){40}} 
\put(237,23){\line(0,1){20}} 
\put(120,0){\line(2,1){121}}
\put(160,0){\line(2,1){40}}
\qbezier(200,20)(220,32)(237,42)
\qbezier(240,0)(240,0)(237,23)
\end{picture}

\noindent 
where any potentially nontrivial Adams differential is 
ruled out by the splitting. 
The condition $j\geq1$ assures that we are in the range for 
a $(8j+6)$-MPT to exist; it takes the form

\begin{equation}\label{tower}
\begin{picture}(10,130)(170,126)
\put(20,250){$B\Spin(8j)$}
\put(55,200){$E_3$} 
\put(55,150){$E_2$} 
\put(55,100){$E_1$}
\put(32,50){$B\Spin$} 
\put(60,243){\vector(0,-1){27}}
\put(60,193){\vector(0,-1){27}} 
\put(60,143){\vector(0,-1){27}}
\put(60,93){\vector(0,-1){27}} 
\put(42,230){$p_4$}
\put(42,180){$p_3$} 
\put(42,130){$p_2$} 
\put(42,80){$p_1$}
\put(77,205){\vector(1,0){25}} 
\put(107,200){$K'_{8j+4}$}
\put(85,207){$k^3$} 
\put(77,155){\vector(1,0){25}}
\put(107,150){$K_{8j+3}\times K_{8j+4}\times K'_{8j+4}$} 
\put(85,157){$k^2$}
\put(77,105){\vector(1,0){25}} 
\put(107,100){$K_{8j+2}\times K_{8j+3}\times K_{8j+4}\times K'_{8j+4}
\times K_{8j+5}$}
\put(85,107){$k^1$} 
\put(80,55){\vector(1,0){23}}
\put(107,50){$K(\bz,8j+1)\times K_{8j+2}\times K_{8j+4}$} 
\put(85,57){$k^0$}
\end{picture}
\end{equation}

\vspace{2.3cm}
\noindent where $K_{m}$ stands for the Eilenberg-Mac$\,$Lane 
space $K(\bz/2,m)$. We let $K(i)$ and $\mu_i: K(i)\times E_i\rightarrow E_i$ 
stand for the fiber and the action, respectively, in the principal 
fibration $p_i$. To conclude the MPT setting, we remark that
a standard calculation gives the following characterizations for
the $k$-invariants:
\begin{eqnarray}
\mu^*_1(k^1_{8j+2}) &=& 1 \otimes k^1_{8j+2} + \sq 2 
\iota_{8j} \otimes 1, \label{A} \\
\mu^*_1(k^1_{8j+3}) &=& 1 \otimes k^1_{8j+3} + \sq 2 \iota_{8j+1}
\otimes 1, \label{B} \\ 
\mu^*_1(k^1_{8j+4}) &=& 1
\otimes k^1_{8j+4} + \sq 1 \iota_{8j+3} \otimes 1 + \sq 2 \sq 1
\iota_{8j+1} \otimes 1, \label{C} \\ 
\mu^*_1(k'^{\, 1}_{8j+4})
&=& 1 \otimes k'^{\, 1}_{8j+4} + \sq 4 \iota_{8j} \otimes 1 + 
\iota_{8j} \otimes w_4, \label{D} \\ 
\mu^*_1(k^1_{8j+5}) &=& 1 \otimes k^1_{8j+5} +
\sq 4 \iota_{8j+1} \otimes 1 + \iota_{8j+1} \otimes w_4, \label{E} \\
\mu^*_2(k^2_{8j+3}) &=& 1 \otimes k^2_{8j+3} + \sq 2 
\iota_{8j+1} \otimes 1, \label{F} \\ 
\mu^*_2(k^2_{8j+4}) &=& 1 \otimes k^2_{8j+4} + \sq 1 
\iota_{8j+3} \otimes 1 + \sq 2
\iota_{8j+2} \otimes 1, \label{G} \\ 
\mu^*_2(k'^{\,2}_{8j+4}) &=& 1 \otimes k'^{\,2}_{8j+4} + \sq 1 
\iota'_{8j+3} \otimes 1 + \sq 2\sq 1
\iota_{8j+1} \otimes 1, \label{H} \\ 
\mu^*_3(k'^{\,3}_{8j+4}) &=& 1 \otimes k'^{\,3}_{8j+4} + \sq 1
\iota'_{8j+3} \otimes 1 + \sq 2 \iota_{8j+2} \otimes 1. \label{I} 
\end{eqnarray}
Here $\iota_t$ stands for the relevant fundamental class in 
the various fibers $K(i)$
(note that we have systematically primed $k$-invariants, Eilenberg-Mac$\,$Lane 
spaces, and  fundamental classes coming from the 
three classes in the right hand tower of height 3 in the chart above).

\smallskip
Step 2. {\em Basic (primary) MPT analysis}.
Since the binomial coefficient
$\binom{-2j-4}{2j+1}$ is even, $k^0_{8j+4}$ is trivial over $\cua \P^{2j+1}$; 
the other two $k^0$-invariants are also trivial over $\cua \P^{2j+1}$
by dimensional reasons. Therefore $g$ lifts in~(\ref{tower}) to a map 
$\ell_1: \cua \P^{2j+1}\to E_1$
whose only possibly nontrivial $k^1$-invariants are $k^1_{8j+4}$ and 
$k'^{\, 1}_{8j+4}$. 
We now show that these two $k^1$-invariants map trivially.

\smallskip
To deal with $\ell_1^*(k'^{\, 1}_{8j+4})$ we first note that 
extending the Adams chart through dimension $8j+7$ extends 
the $(8j+6)$-MPT to an $(8j+8)$-MPT with only one extra 
$k^0$-invariant, $k^0_{8j+8}$. Since 
$\binom{-2j-4}{2j+2}$ is even, $k^0_{8j+8}$ is trivial over 
$\cua \P^{2j+2}$ and, thus, the lifting $\ell_1: \cua \P^{2j+1}
\to E_1$ extends to a lifting 
$\cua \P^{2j+2}\to E_1$. In this extended MPT, there 
is a $k^2$-invariant, $k^2_{8j+7}$, defined by the relation 
$$(\sq 4 + w_4)k'^{\, 1}_{8j+4} + (\sq 6 + w_6)k^1_{8j+2} = 0$$ in 
$E_1$. This implies $(\sq 4 + w_4)\ell_1^*(k'^{\, 1}_{8j+4}) = 0$. 
Thus, assuming $\ell_1^*(k'^{\, 1}_{8j+4})$ is the nontrivial class 
in $H^{8j+4}(\cua \P^{2j+2};\mathbb{Z}/2)$ yields a 
contradiction since $\sq 4$ of this class is nontrivial, 
but $w_4 = 0$ for $g$.

\smallskip
To deal with $\ell_1^*(k^1_{8j+4})$ we need the fact that $g$ lifts to 
$B\Spin(8j+1)$. This follows from a straightforward application 
of~\cite[Theorem~1.3(b)]{dama75}---in a similar way to our argument
for~(\ref{l1m2}), but this time with quaternions
instead of complex numbers. With this information at hand, 
we consider the following diagram of MPT's, where the left hand tower 
is~(\ref{tower}), the right hand tower is the MPT for 
$B\Spin(8j+1)\to B\Spin$, and the map of MPT's is induced by 
the canonical map $j\colon B\Spin(8j)\to B\Spin(8j+1)$. 

\begin{equation}\label{diagr}
\begin{picture}(50,110)(40,90)
\def\elt{\circle*{1}}
\def\mp{\multiput}
\mp(7,145)(0,5){3}{\elt} 
\mp(107,145)(0,5){3}{\elt} 
\put(-20,0){$B\Spin$}
\put(100,0){$B\Spin$}
\put(0,50){${E_1}$}
\put(100,50){$E'_1$}
\put(0,100){${E_2}$}
\put(100,100){$E'_2$}
\mp(7,40)(100,0){2}{\vector(0,-1){20}}
\mp(7,90)(100,0){2}{\vector(0,-1){20}}
\mp(7,140)(100,0){2}{\vector(0,-1){20}}
\mp(7,180)(100,0){2}{\vector(0,-1){20}}
\put(-30,190){\small ${B\Spin}(8j)$}
\put(70,190){\small ${B\Spin}(8j+1)$}
\mp(30,6)(100,0){1}{\line(1,0){60}}
\mp(30,3)(100,0){1}{\line(1,0){60}}
\mp(30,55)(100,0){1}{\vector(1,0){60}}
\mp(30,105)(100,0){1}{\vector(1,0){60}}
\put(47,195){\vector(1,0){19}}
\put(52,60){\scriptsize $j_1$}
\put(52,200){\scriptsize $j$}
\put(52,110){\scriptsize $j_2$}
\end{picture}
\end{equation}

\vspace{2.4cm}\noindent 
Moreover, the effect in homotopy (up to dimension $8j+3$) 
of the mapping between fibers can be read in the 
Adams charts (see~\cite[Tables~8.1--8.2 for even $j$, and 8.9--8.10
for odd $j$]{meta})

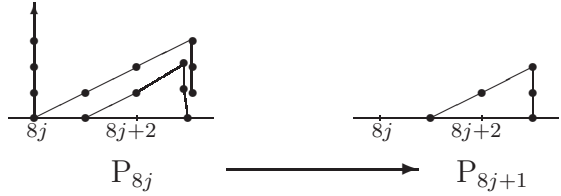
\begin{picture}(10,110)(-50,-43)
\setlength{\unitlength}{.17mm}
\def\elt{\circle*{5.2}}
\def\mp{\multiput}
\put(270,-40){{\thicklines\vector(1,0){150}}}
\put(0,0)
{\begin{picture}(0,0)
\put(100,0){\line(1,0){160}} 
\mp(120,0)(0,20){4}{\elt} 
\put(120,0){\vector(0,1){90}}
\mp(160,0)(0,20){2}{\elt} 
\mp(200,20)(0,20){2}{\elt} 
\mp(244,20)(0,20){3}{\elt} 
\mp(237,23)(0,20){2}{\elt} 
\mp(240,0)(0,20){1}{\elt} 
\put(114,-15){\scriptsize $8j$}
\put(200,-2){\line(0,1){4}} 
\put(176,-15){\scriptsize $8j{+}2$} 
\put(243,20){\line(0,1){40}} 
\put(237,23){\line(0,1){20}} 
\put(120,0){\line(2,1){121}}
\put(160,0){\line(2,1){40}}
\qbezier(200,20)(220,32)(237,42)
\qbezier(240,0)(240,0)(237,23)
\put(180,-50){$\P_{8j}$}
\end{picture}}
\put(270,0)
{\begin{picture}(0,0)
\put(100,0){\line(1,0){160}} 
\mp(160,0)(0,20){1}{\elt} 
\mp(200,20)(0,20){1}{\elt} 
\put(114,-15){\scriptsize $8j$} 
\put(200,-2){\line(0,1){4}} 
\put(176,-15){\scriptsize $8j{+}2$} 
\put(240,0){\line(0,1){40}} 
\put(160,0){\line(2,1){80}}
\put(180,-50){$\P_{8j+1}$}
\mp(240,0)(0,20){3}{\elt}
\mp(120,-2)(40,0){1}{\line(0,1){4}} 
\end{picture}}
\end{picture}

\noindent
In particular, the $(8j+4)$-dimensional $k^1$-invariant on 
the right hand side 
in~(\ref{diagr}) maps into $k^1_{8j+4}$ in~(\ref{tower}).
Now, since $g$ lifts to $B\Spin(8j+1)$ the composite $j_1 
\ell_1$ has trivial $k^1$-invariants (with trivial indeterminacies). 
It follows that $\ell_1^*({k}^1_{8j+4})=0$.

\smallskip
We have proved:
\begin{lema}\label{provisional}
The map $g$ lifts in~(\ref{tower}) to a map $\ell_2: \cua 
\P^{2j+1} \rightarrow E_2$.
\end{lema}

\begin{nota}\label{trivial}{\em
Let $\widetilde{\ell}_2=\ell_2 \,f$. Note that 
$\widetilde{\ell}_2^*(k^2_{8j+3})$ is trivial by dimensional reasons, 
whereas~(\ref{G}) and the 
fact that $\sq2$ acts nontrivially on $H^{8j+2}(\com \P^{4j+3};\mathbb{Z}/2)$ 
imply we can kill $\widetilde{\ell}_2^*(k^2_{8j+4})$ 
by primary indeterminacy. Moreover,~(\ref{H}) shows that this killing of
$\widetilde{\ell}_2^*(k^2_{8j+4})$ does not modify 
$\widetilde{\ell}_2^*(k'^{\, 2}_{8j+4})$. Also,~(\ref{I}) shows that the 
$k^3$-invariant of any possible lifting $\widetilde{\ell}_3:\com 
\P^{4j+3}\rightarrow E_3$ of $gf $ can be killed by primary indeterminacy. 
}\end{nota}

Step 3. {\em Secondary indeterminacy}.
We now prove the next result, which coupled with Remark~\ref{trivial}
implies Theorem~\ref{fin}.

\begin{proposicion}\label{BB}
There is a lifting $\widetilde{\ell}'_2: \com \P^{4j+3}\rightarrow E_2$ of $gf$ 
in~(\ref{tower}) that has trivial 
$k'^{\, 2}_{8j+4}$-invariant.
\end{proposicion}

\begin{proof} We show that a nontrivial $\widetilde{\ell}_2^*(k'^{\, 2}_{8j+4})$
can be corrected through secondary indeterminacy. To this end,
we first need to make a slight adjustment in our MPT:
We already noticed that $w_4=0$ for $g$, so that this map factors 
as $$\mathbb{H}\P^{2j+1}\stackrel{g'}{\to}F\to B\Spin,$$ where
$F$ is the fiber of $w_4: B\Spin\rightarrow K_4$. Note that, since 
$\pi_4(B\Spin)= \bz$, $F$ is still $3$-connected (not $4$-connected), 
but we have killed $w_4$ in $F$. We will prove the required condition for 
$g' f$ with~(\ref{tower}) replaced by 
the induced MPT over $F$; therefore we extend accordingly our 
notations for~(\ref{tower}) to this MPT over $F$. In particular
Remark~\ref{trivial} gives us a corresponding lifting 
\begin{equation}\label{l2}
\widetilde{\ell}_2\colon\mathbb{C}\P^{4j+3}\to E_2
\end{equation}
of $g' f$, and 
we will show how to alter (if needed) its $k'^{\, 2}_{8j+4}$-invariant 
by secondary indeterminacy coming from the class 
$y^{4j}\in H^{8j}(\mathbb{C}\P^{4j+3};\mathbb{Z}/2)$,
where $y$ stands for the generator in 
$H^2(\mathbb{C}P^{4j+3};\mathbb{Z}/2)$---although 
not needed, we remark that
this is in fact the only way to correct the problematic 
$k'^{\, 2}_{8j+4}$-invariant. The next considerations are preparatory.

\smallskip
We start by expanding the MPT over $F$ (up to stage $E_2$) to the 
commutative diagram

{\scriptsize
\begin{picture}(100,163)(-70,-13)
\put(254,70){$K_{8j{+}2}{\times} K_{8j{+}3}{\times} K_{8j{+}4}{\times} 
K'_{8j{+}4}{\times} K_{8j{+}5}$}
\put(235,73){\vector(1,0){14}}
\put(215,70){$E_1$}
\put(215,120){$E_2$}
\put(215,-5){$F$}
\put(203,73){\vector(1,0){10}}
\put(222,113){\vector(0,-1){28}}
\put(222,63){\vector(0,-1){52}}
\put(224,97){$p_2$}
\put(224,24){$p_1$}
\put(191,123){\vector(1,0){20}}
\put(80,70){$K(\bz,8j){\times} K_{8j+1}{\times} K_{8j+3}$}
\put(-59,120){$K_{8j{+}1}{\times} K_{8j{+}2}{\times} K_{8j{+}3}{\times} 
K'_{8j{+}3}{\times} K_{8j{+}4}$}
\put(136,123){\vector(1,0){28}}
\put(184,96){$\rho$}
\put(180,115){\vector(0,-1){35}}
\put(55,93){$X$}
\put(40,40){$K(\bz,8j)$}
\put(60,88){\vector(0,-1){33}}
\put(52,70){$p$}
\put(260,40){$K_{8j{+}2}{\times} K'_{8j{+}4}$}
\multiput(93,45)(22,0){7}{\line(1,0){9}}
\put(245,45){\vector(1,0){10}}
\put(110,29){$(\sq 2,\sq 4)$}
\put(296.12,49){\vector(1,1){14}}
\qbezier(292,53)(291,45)(296,49)

\put(0,0){\begin{picture}(0,0)(220,-2.5)
\put(296.12,49){\vector(1,1){14}}
\qbezier(292,53)(291,45)(296,49)
\end{picture}}

\put(75,97){\vector(4,1){90}}
\put(100,107){$\gamma$}
\put(-60,93){$K_{8j{+}1}{\times} K'_{8j{+}3}$}
\put(-405,53){
\begin{picture}(0,0)
\qbezier(390,50)(408,60)(408,60)
\put(410,61){\vector(2,1){0}}
\qbezier(386,56)(382,47)(390,50)
\end{picture}}
\put(12,97){\vector(1,0){32}}
\put(28,100){$j$}
\put(-30,10){$\com \P^{4j+3}$}
\put(168,120){$F_{0,2}$}
\put(22,17){\scriptsize $y^{4j}$}
\put(10,17){\vector(2,1){36}}
\put(10,20){\vector(1,2){36}}
\put(15,53){$\alpha$}
\end{picture}}

\noindent Here the square involving $\rho$ and $p_2$ is a pull-back, so that 
$F_{0,2}$ can be thought of as 
the common fiber of (a) the composite $p_1 \,p_2$, and (b)
the middle horizontal composite 
$$K(\bz,8j){\times} K_{8j+1}{\times} K_{8j+3} \to E_1 \to 
{K_{8j{+}2}{\times} K_{8j{+}3}{\times} K_{8j{+}4}{\times} 
K'_{8j{+}4}{\times} K_{8j{+}5}}.$$
Likewise, $p\colon X \to K(\bz,8j)$ is the fiber inclusion of the 
dashed map $(\sq 2,\sq 4)$. 
The map $\gamma : X \to F_{0,2}$ exists in view of the commutativity of 
the square involving the $(\sq 2,\sq 4)$ map 
(see~(\ref{A})---(\ref{E})). 
The indicated lifting $\alpha$ of $y^{4j}$ exists since 
$(\sq 2,\sq 4)(y^{4j})=0$. 

\smallskip
Next we describe suitable variations of the principal 
actions for $p_1$ and $p_2$ in the MPT over $F$.
Recall $K(1)$ and $K(2)$ denote, respectively, 
$K(\bz,8j){\times} K_{8j+1}{\times} K_{8j+3}$ and 
$K_{8j{+}1}{\times} K_{8j{+}2}{\times} 
K_{8j{+}3}{\times} K'_{8j{+}3}{\times} K_{8j{+}4}$. We define 
$\mu^{\prime}_1:F_{0,2} {\times} E_1 \rightarrow E_1$ to be the
composite $$F_{0,2} {\times} E_1 \stackrel{\rho {\times} 1}{\longrightarrow}
K(1) {\times} E_1 \stackrel{\mu_1}{\longrightarrow} E_1.$$  We then have the
diagram

\begin{picture}(50,170)(-113,27)
\put(0,170){$K_{8j+1}{\times} K'_{8j+3} {\times} E_2$}
\put(126,177){\oval(6,7)[l]}
\put(126,173.5){\vector(1,0){25}}
\put(226,173.5){\vector(1,0){35}}
\put(238,179.5){\scriptsize $\mu_2$}
\put(225,150){\scriptsize $\mu'_2$}
\put(157,170){$K(2) {\times} E_2$}
\put(265,170){$E_2$}
\put(55,160){\vector(0,-1){30}}
\put(31,142){\scriptsize $j{\times}1$}
\put(190,160){\vector(0,-1){30}}
\put(273,160){\vector(0,-1){135}}
\put(30,115){$X\times E_2$}
\put(85,120){\vector(1,0){66.5}}
\put(110,109){\scriptsize $\gamma{\times}1$}
\multiput(80,125)(9.2,2){19}{$\cdot$}
\put(110,54){\scriptsize $\gamma{\times}1$}
\put(157,115){$F_{0,2}\times E_2$}
\put(55,105){\vector(0,-1){30}}
\put(25,87){\scriptsize $1{\times}p_2$}
\put(160,87){\scriptsize $1{\times}p_2$}
\put(278.5,87){\scriptsize $p_2$}
\put(190,105){\vector(0,-1){30}}
\put(258,167.3){\vector(4,1){0}}
\put(30,60){$X\times E_1$}
\put(85,65){\vector(1,0){66.5}}
\put(157,60){$F_{0,2}\times E_1$}
\put(55,50){\vector(0,-1){27}}
\put(31,32){\scriptsize $p{\times}1$}
\put(190,50){\vector(0,-1){27}}
\put(166,32){\scriptsize $\rho{\times}1$}
\put(10,5){$K(\bz,8j) {\times} E_1$}
\put(105,10){\vector(1,0){47.5}}
\put(105,13.7){\oval(6,7)[l]}
\put(157,5){$K(1) {\times} E_1$}
\put(265,5){$E_1$}
\put(226,10){\vector(1,0){35}}
\put(238,0){\scriptsize $\mu_1$}
\put(240,40){\scriptsize $\mu'_1$}
\put(207,50){\vector(2,-1){55}}
\end{picture}

\vspace{1cm}\noindent
The dashed map $\mu^{\prime}_2: X {\times} E_2
\longrightarrow E_2$, rendering a commutative diagram, 
exists since the composite
$$
X \times E_2 \stackrel{\gamma\times p_2}{-\!\!\!-\!\!\!\lra} F_{0,2} 
\times E_1 \stackrel{\mu_1'}{\lra} E_1 \stackrel{k^1}{\lra} K_{8j{+}2} \times 
K_{8j{+}3}\times K_{8j{+}4}\times K'_{8j{+}4}\times K_{8j{+}5}
$$
is null-homotopic as shown by the following chase of classes (where
the symbol ``$\hookrightarrow$'' stands for 
the map induced by the bottom inclusion in the diagram above):
{\small
\begin{eqnarray*}
& & k^1_{8j+2}\stackrel{\mu_1}{\longrightarrow} 1 {\otimes} k^1_{8j+2} + \sq 2 
\iota_{8j} {\otimes} 1 \hookrightarrow 1 {\otimes} k^1_{8j+2} + \sq 2 
\iota_{8j} {\otimes} 1 \stackrel{p {\times} 1}{\longrightarrow} 1 {\otimes} 
k^1_{8j+2}
\stackrel{1 {\times} p_2}{\longrightarrow} 0 \\ & &
k^1_{8j+3}\stackrel{\mu_1}{\longrightarrow} 1 {\otimes} k^1_{8j+3} + \sq 2 
\iota_{8j+1} {\otimes} 1 \hookrightarrow 1 {\otimes} k^1_{8j+3}
\stackrel{p {\times} 1}{\longrightarrow} 1 {\otimes} k^1_{8j+3}
\stackrel{1 {\times} p_2}{\longrightarrow} 0 \\ & &
k^1_{8j+4}\stackrel{\mu_1}{\longrightarrow} 1 {\otimes} k^1_{8j+4} + 
\sq 1 \iota_{8j+3} {\otimes} 1 + \sq 2 \sq 1
\iota_{8j+1} {\otimes} 1 \hookrightarrow 1 {\otimes} k^1_{8j+4}
\stackrel{p {\times} 1}{\longrightarrow} 1 {\otimes} k^1_{8j+4}
\stackrel{1 {\times} p_2}{\longrightarrow} 0 \\ & &
k'^{\, 1}_{8j+4}\stackrel{\mu_1}{\longrightarrow} 1 {\otimes} 
k'^{\, 1}_{8j+4} + \sq 4 \iota_{8j} {\otimes} 1  \hookrightarrow 1 
{\otimes} k'^{\, 1}_{8j+4} + \sq 4 \iota_{8j} {\otimes} 1  
\stackrel{p {\times} 1}{\longrightarrow} 1 {\otimes} k'^{\, 1}_{8j+4}
\stackrel{1 {\times} p_2}{\longrightarrow} 0 \\ & &
k^1_{8j+5}\stackrel{\mu_1}{\longrightarrow} 1 {\otimes} k^1_{8j+5} + 
\sq 4 \iota_{8j+1} {\otimes} 1  \hookrightarrow 1 {\otimes} k^1_{8j+5}
\stackrel{p {\times} 1}{\longrightarrow} 
1 {\otimes} k^1_{8j+5}\stackrel{1 {\times} p_2}{\longrightarrow} 0
\end{eqnarray*}}
Note that, in the last two rows, summands with a $w_4$ are trivial by
construction of $F$. (For readers familiar with~\cite{tom}:
the point of this explicit calculation is to make sure that all terms 
in the chase lie in wedge portions of the relevant spaces, a point argued
in~\cite{tom} just from easy dimensional reasons.)

\smallskip
We are finally in a position to explicitly indicate how to kill (if needed),
through secondary indeterminacy, a nontrivial 
$k'^{\, 2}_{8j+4}$-invariant for~(\ref{l2}). Start with the composite 
$$\widetilde{\ell}_1'=\mu_1^{\prime}((\gamma \alpha) {\times}(p_2 
\widetilde{\ell}_2)):\com \P^{4j+3} \to E_1$$---the 
(primary) modification of $\widetilde{\ell}_1=
p_2 \widetilde{\ell}_2$ by the class $y^{4j}$. From the 
diagram that defines $\mu_2^{\prime}$ we have, first, that an explicit lifting
of $\widetilde{\ell}_1'$ is given by the map
$$\widetilde{\ell}_2'=\mu_2^{\prime}( \alpha {\times} \widetilde{\ell}_2):
\com \P^{4j+3} \to E_2$$ and, second, that 
equation~(\ref{H}) translates, by dimensional reasons, into 
$$(\mu^{\prime}_2)^*(k'^{\,2}_{8j+4}) = 1 \otimes k'^{\,2}_{8j+4} + a 
\otimes 1 + \iota_{8j} \otimes b$$ where $a \in H^{8j+4}(X;\mathbb{Z}/2)$ 
and $b \in H^4(E_2;\mathbb{Z}/2)$, and where the former element satisfies
\begin{equation}\label{cdlos}
j^*(a) = \sq 1 \iota'_{8j+3} \otimes 1 + \sq 2\sq 1 \iota_{8j+1} 
\otimes 1.
\end{equation}
In particular, the $k'^{\,2}_{8j+4}$-invariant of $\widetilde{\ell}_2'$ 
is computed as
\begin{eqnarray*}
(\widetilde{\ell}_2')^*(k'^{\,2}_{8j+4})
&=& ( \alpha {\times} \widetilde{\ell}_2)^*(1 \otimes k'^{\,2}_{8j+4} + 
a \otimes 1 + \iota_{8j} \otimes b) \\ 
&=&
\widetilde{\ell}_2^*(k'^{\,2}_{8j+4}) + \alpha^*(a) + 
\alpha^*(\iota_{8j}) \cdot \widetilde{\ell}_2^*(b).
\end{eqnarray*}
Therefore we will be done once we establish the two relations
\begin{equation}\label{theend}
\alpha^*(a)=y^{4j+2}\mbox{ \ \ and \ \ } \widetilde{\ell}_2^*(b)=0.
\end{equation}

By construction
$\alpha^*(a) \in \Phi(y^{4j})$, where $\Phi$ is the secondary operation 
determined by the class $a \in H^{8j+4}(X;\mathbb{Z}/2)$ (subject to~(\ref{cdlos}))
in the diagram

\begin{picture}(50,115)(20,-23)
\put(100,0){$\com \P^{4j+3}$}
\put(220,0){$K(\bz,8j)$}
\put(354,0){$K_{8j+2}{\times} K_{8j+4}$}
\put(110,60){$K_{8j+1}{\times} K_{8j+3}$}
\put(240,60){$X$}
\put(300,60){$K_{8j+4}$}
\put(155,3){\vector(1,0){57}}
\put(285,3){\vector(1,0){63}}
\put(255,63){\vector(1,0){37}}
\put(205,63){\vector(1,0){28}}
\put(247,50){\vector(0,-1){32}}
\put(155,14){\vector(2,1){80}}
\put(177,-12){\small $y^{4j}$}
\put(285,-12){\small $(\sq2,\sq4)$}
\put(270,70){\small $a$}
\put(215,70){\small $j$}
\put(180,35){\small $\alpha$}
\end{picture}

\noindent
This operation is associated to the relation 
$(\sq 2 \sq 1) \cdot \sq 2 + \sq 1 \cdot \sq 4 = 0$
(for our purposes this is a relation on elements coming from 
integral classes, so that $\sq1$ acts trivially on 
them). But in~\cite[Theorem~A]{GMM} 
it is shown that $\Phi(y^{4j})=y^{4j+2}$ for odd $j$
(with trivial indeterminacy), giving the first equality in~(\ref{theend}). The 
second equality in~(\ref{theend}) is much simpler. 
Assume $b\not=0$. Since $H^4(B\Spin;\bz)=\bz$ generated by $p_1/2$, 
it follows that $H^4(F;\bz)=\bz$ generated by $p_1$.
Then the mod 2 reduction of $p_1$ must correspond to $b$ under 
$$\mathbb{Z}/2=H^4(F;\mathbb{Z}/2)\stackrel{p_1^*}{\approx}H^4(E_1;\mathbb{Z}/2)
\stackrel{p_2^*}{\approx}H^4(E_2;\mathbb{Z}/2).$$ Thus, 
$\widetilde{\ell}_2^*(b) = p_1(-4(j+2)H) =  (-4(j+2))y^2 = 0$.
\end{proof}

We close this section with a final word about the optimality of the 
obstruction theory methods. As we have seen, 
the basis for Lemma~\ref{feeding} comes (rather indirectly) from 
the existence of diagrams of the form~(\ref{diagram}), settled here 
by means of modified Postnikov towers---and 
by direct linear algebra constructions
in~\cite{MM}. Of course, the smaller $d$ one can use in~(\ref{diagram}),
the better are the results produced by the inductive rounds
in Section~\ref{newcomp}. Now, in Remark~\ref{power}
we observed the optimality of some of the embeddings in the conclusion
of Lemma~\ref{feeding} {\em for $e=1$}. In particular, we get the optimality of 
the corresponding $d$ used in~(\ref{diagram}). The point we want to stress 
here is that the latter optimality holds {\em for any $e$}, in view of
the compatibility of the embeddings $L_{n,e}\hookrightarrow L_{m,e}$ under
the canonical projections $L_{t,e}\to L_{t,e+1}$. The net outcome 
to remark is that the homotopy lifting in Proposition~\ref{fin} 
turns out to be optimal when $j$ is one less than a power of 2---the critical 
$\mu=2$ case in~(\ref{tom}). Together with the corresponding embeddings of
lens space, this infinite family of optimal liftings 
justifies the several pages of MPT manipulations.

\section{Immersion vs.~embedding dimension}\label{comp}
In this final section of the paper we focus 
on certain subtle points arising when comparing the behavior
of the (immersion and) embedding dimensions of $2^e$-torsion lens spaces,
as $e$ varies. The case $t=2^e$ of the following easy observation was 
mentioned in~\cite{nonimm} as a convenient way
to take advantage of known immersion results for (real and complex) 
projective spaces when studying 
the immersion problem for lens spaces.

\begin{hecho}\label{comparar}{
Any codimension-$k$ Euclidean immersion for 
$\mathbb{C} \P^n$ brings for free a co\-dimension-$k$ Euclidean immersion 
for any $t$-torsion lens space $L^{2n+1}(t)$. Likewise, when $t'$ divides $t$,
any codimension-$k$ Euclidean immersion of $L^{2n+1}(t)$ 
brings for free a codimension-$k$ Euclidean immersion of $L^{2n+1}(t')$.
}\end{hecho}

Unfortunately, the first (and potentially the second) statement(s) 
in Fact~\ref{comparar} readily fails when ``immersion'' is replaced 
by ``embedding''. For instance, $\mathbb{C} \P^1$ 
embeds in $\mathbb{R}^3$, but as observed in Remark~\ref{n12}, no 
$3$-dimensional lens space embeds in $\mathbb{R}^4$. 
After discussing how this problem arises and how it could be mended,
we deal, in Remark~\ref{utilidad}  below, with the 
potential usefulness of such a possibly mended embedding analogue.
For the time being we note that this initial problem could just as well be  
a facet of the following very peculiar situation: 

\begin{hecho}\label{parallel}{\em
{\it With the
exception of $S^1$ and $\P^7$, which are well known to 
be parallelizable manifolds, $3$-dimensional lens spaces are the only 
parallelizable manifolds among lens spaces.} In more detail, 
Lemma~\ref{spingen} and Proposition~\ref{spin3} (below) give the
parallelizability of all $3$-dimensional lens spaces
(the same argument gives the well-known parallelizability of
any compact orientable $3$-dimensional manifold), 
whereas it is proved in~\cite{litang} that, except for $\P^7$ and 
the following list of stably parallelizable (but not parallelizable) manifolds,
no other lens space of dimension $\ge 5$ is stably parallelizable.
The exceptional lens spaces
are $L^{2p-1}(p)$, for $p$ an odd prime.
}\end{hecho}

\begin{proposicion}\label{spin3}
Any $3$-dimensional spin bundle over a $3$-dimensional complex is trivial.
\end{proposicion}
\begin{proof} This follows from the fact that $B\Spin(3)=
BS^3$ is $3$-connected. \end{proof}

Here is an explicit way of looking (within Haefliger's 
metastable range) at the problems arising in the quest for an embedding 
analogue of the second statement of
Fact~\ref{comparar}.
Consider the configuration space $F_{\mathbb{Z}/2^e}(S^{2n+1},2)$
consisting of pairs in $S^{2n+1}\times S^{2n+1}$ generating 
different $\mathbb{Z}/2^e$ orbits. As observed 
in~\cite[Lemma~5.1]{symmotion}, for $k\ge3(n + 1)$, Haefliger's
characterization~\cite[Th\'eor\`eme $1'$]{Haefligerstable} 
for the existence of an embedding $L^{2n+1}(2^e)\subset R^k$ 
decodes into the existence of a map $\alpha\colon 
F_{\mathbb{Z}/2^e}(S^{2n+1},2)\to S^{k-1}$ satisfying 
\begin{equation}\label{haefliger}
\alpha(\omega x,y)=\alpha(x,y)=\alpha(x,\omega y)
\mbox{ \ \ and \ \ }\alpha(x,y)=-\alpha(y,x)
\end{equation}
for $(x,y)\in F_{\mathbb{Z}/2^e}(S^{2n+1},2)$, and $\omega\in\mathbb{Z}/2^e$.
Now, although any map satisfying~(\ref{haefliger}) for all 
$\omega\in\mathbb{Z}/2^e$ will certainly satisfy the same requirements for all 
$\omega\in\mathbb{Z}/2^{e-1}$, the strict inclusion 
$F_{\mathbb{Z}/2^e}(S^{2n+1},2)\subset F_{\mathbb{Z}/2^{e-1}}(S^{2n+1},2)$
can only be used to interpret, as an extension problem, 
the embedding analogue of the second statement in Fact~\ref{comparar}.

\smallskip
Despite the above problems, it might still be the case that a restricted  
embedding version of Fact~\ref{comparar} holds.
The following considerations are meant to shed evidence on such a possibility.
To this end, it will be convenient to use the notation 
$\E(M)$ for the smallest
dimension of the Euclidean space where the given manifold $M$ can be embedded. 
The information we use below about these numbers 
is taken from~\cite{dontablas}, in the case of real projective spaces, and 
from~\cite{AH,amiya}, in the case of complex projective spaces.

\smallskip
First of all, the considerations in Fact~\ref{parallel} indicate that,
in asking whether an embedding analogue 
of Fact~\ref{comparar} holds,
it might be fair to exclude the case $n=1$.
In fact, we should exclude the whole family $\alpha(n)=1$ as, in such a case,
$\E(\mathbb{C} \P^{n})=4n-1$, but according to Theorem~\ref{enepot},
$\E(L^{2n+1}(2^e))=4n+1$ for all 
$e\ge 1$---one dimension higher than what 
the embedding analogue of Fact~\ref{comparar} would anticipate.

\smallskip
Assuming now $\alpha(n) > 1$, there does not seem to be an 
immediate problem, at least at the outset (but perhaps mainly due to a lack of 
information), for a potential embedding analogue of Fact~\ref{comparar}. 
For instance, when $n=6$ (the first such case with even $n$),
$\E(\mathbb{C} \P^6)\in\{21,22\}$ whereas $\E(\P^{13})\in\{22,23\}$ 
---three out of the four possibilities being 
compatible with a potential embedding analogue of Fact~\ref{comparar}. 
 
\smallskip
But for odd $n$ the situation is much nicer, and
the problem discussed above for $\alpha(n)=1$ is replaced by the following
families\footnote{Only projective spaces are considered in this sample;
trying to include lens spaces gives an excellent motivation for  
sharpening the embeddings in Theorem~\ref{niam}, and in particular for
trying to close up the small
gaps indicated in the second paragraph after that theorem.} 
that are fully compatible with
a possible embedding version of Fact~\ref{comparar}. 
Indeed:
\begin{itemize}
\item[{(a)}] For 
$n=2^e+1\ge 3$, $\E(\mathbb{C} \P^n)=2^{e+2}+1$, whereas
$\E(\P^{2n+1})\le 2^{e+2}+2$ when $e\ge 2$.

\vspace{-.3mm}\item[{(b)}] For 
$n=2^e+3\ge 7$, $\E(\mathbb{C} \P^n)\in 2^{e+2}+\{7,8\}$, 
whereas $\E(\P^{2n+1})=2^{e+2}+7$ when $e\ge 3$.

\vspace{-.3mm}\item[{(c)}] For 
$n=2^e+5\ge 13$, both $\E(\mathbb{C} \P^n)$ and $\E(\P^{2n+1})$ 
are of the form $2^{e+2}+\{15,16\}$.

\vspace{-.3mm}\item[{(d)}] For 
$n=2^e+7\ge 15$, $\E(\mathbb{C} \P^n)\in 2^{e+2}+\{21,22,23\}$, 
whereas $\E(\P^{2n+1})\le 2^{e+2}+22$ when $e\ge 4$.

\vspace{-.3mm}\item[{(e)}] For 
$n=2^e+9\ge 25$, both $\E(\mathbb{C} \P^n)$ and $\E(\P^{2n+1})$
are of the form $2^{e+2}+\{31,32\}$.
\end{itemize}

\begin{nota}\label{utilidad}{\em
With the above as indirect evidence, and in order to illustrate its 
potential usefulness, we observe that an embedding 
version of the first statement in Fact~\ref{comparar} 
would imply, in view of Theorem~\ref{main},  that
the ($4n-\alpha(n)-1$)-dimensional Euclidean embedding 
of $\mathbb{C} \P^n$ (for odd $n>1$) in~\cite{amiya} would not only 
be optimal for $\alpha(n)=2$, but would also produce 
a corresponding ($4n-\alpha(n)$)-dimensional Euclidean
embedding of any lens space $L^{2n+1}(2^e)$, optimal when $e\ge \alpha(n)=2$,
in view of Theorem~\ref{main}.
}\end{nota}

In the direction of the potential optimality of the embeddings 
in~\cite{amiya}, we observe
that the immersion (and consequently embedding) 
dimension of $\mathbb{C} \P^3$ is $9$. On the one hand,
Remark~\ref{utilidad} already 
mentions the embedding $\mathbb{C}\P^3\subset\mathbb{R}^{9}$. 
On the other, the normal bundle $\nu$ of a hypothetical 
codimension-2 immersion of $\mathbb{C} \P^3$ 
would necessarily be the realification of a complex line bundle. 
Over $\mathbb{C} \P^3$ this would be of the form
$\nu=(H^d)_{\mathbb{R}}$, the realification of the $d$-th power 
of the Hopf bundle ($d\in\mathbb{Z}$). 
In particular, the 
first Pontryagin class of $\nu$ would easily be computed (using for 
instance~\cite[Corollary~15.5]{MS}) to be 
$d^2a^2$, where $a\in H^2(\mathbb{C} \P^3;\mathbb{Z})$ 
stands for the Euler class of $H$. 
However this is incompatible with
the calculation of the (tangential) Pontryagin class 
$p_1(\mathbb{C} \P^3)=4a^2$ 
in~\cite[page 178]{MS} and the fact that 
$p_1(\mathbb{C} \P^3)+p_1(\nu)=0$ (recall that $\mathbb{C} \P^3$ 
is torsion-free).

\begin{nota}\label{utilidad2}{\em
As for the usefulness of a possible embedding analogue of the second statement
in Fact~\ref{comparar}, we remark that this would immediately 
imply the validity of Theorem~\ref{enepot} from the known validity for $e=1$.
}\end{nota}

\begin{nota}\label{eje7}{\em
$\P^7$ is an important benchmark giving an exceptional situation 
both for Fact~\ref{parallel} and Theorem~\ref{main}. 
This manifold is known not to embed in $\mathbb{R}^8$~(\cite{Hantzsche}), 
but has currently best known embedding in 
$\mathbb{R}^{12}$~(\cite{mah64}), and even in $\mathbb{R}^{10}$, if the 
embedding is only required to be 
piecewise linear~(\cite{Rees}). 
It is interesting to observe that a (topological, at least) embedding 
\begin{equation}\label{potential}
L^7(2^e)\subset\mathbb{R}^{10}
\end{equation}
could potentially be obtained for any $e$ by quotienting out
the action of $\mathbb{Z}/2^e$
both in the fiber and total space of the Hopf fibration $S^3\to S^7\to S^4$.
Indeed, this produces a fibration $L^3(2^e)\to L^7(2^e)\to S^4$ 
that, when restricted to the hemispheres of $S^4$, leads to a 
splitting of $L^7(2^e)$ into two parts, each fiber homeomorphic 
to $L^3(2^e)\times D^4$ and, therefore 
(Remark~\ref{n12}), embedding in $\mathbb{R}^9$. Following
the philosophy in the previous section, if 
these embeddings were (homotopy) compatible in the common intersection, 
then there would be a reasonable
chance of getting the embedding in~(\ref{potential})---Rees' 
embedding $P^7\subset\mathbb{R}^{10}$ would seem to suggest that such 
compatibility can be attained for $e=1$. But we have not 
been able to make this idea work for general $e$ (notice that~(\ref{potential})
would be fully 
compatible with the possibilities discussed in Remark~\ref{utilidad}). 
Unfortunately, the relevance of a possible embedding~(\ref{potential}) within 
the inductive proof of Theorem~\ref{niam} is admittedly limited;
the reader will check that the only situation where~(\ref{potential}) 
would produce a better embedding than those described
in Theorem~\ref{niam} (besides the case $n=3$ in the second column of
Table~\ref{tabla1}) is for improving $(e,n,d)=(\leq2,7,26)$, the 
special triples in Theorem~\ref{niam}, to $(e,n,d)=
(\leq2,7,25)$---pretty much as described at the beginning of 
Remark~\ref{power}.
}\end{nota}

We have just mentioned the possibility of extending Rees' topological 
embedding $\P^7\subset\mathbb{R}^{10}$ to 7-dimensional $2^e$-torsion 
lens spaces.  
In this connection, it is well to keep in mind Sanderson's conjecture
in~\cite{sanderson} that the smooth embedding dimension of $\P^7$ is $11$.

\begin{ejemplo}\label{ejegen}{\em
Consider lens spaces $L^{2n+1}(2^e)$ with $n=2^t+1$ for $t\geq2$ (the case
of $t\leq1$ is described in Remark~\ref{n12}). For $e\geq2$, 
i.e.~high-torsion lens spaces, the 1-dimension
gap we leave for $\E(L^{2n+1}(2^e))=2^{t+2}+\delta(e)$, with 
$\delta(e)\in\{2,3\}$, is a shifted version of the known gap
for $e=1$: according to~\cite{dontablas} the best current 
information gives $\E(\P^{2n+1})=2^{t+2}+\delta(1)$, with $\delta(1)\in\{1,2\}$.
}\end{ejemplo}

\bigskip\bigskip

Jes\'us Gonz\'alez\quad {\tt jesus@math.cinvestav.mx}

\medskip
{\sl Departamento de Matem\'aticas, CINVESTAV--IPN (on sabbatical leave)

Apartado Postal 14-740 M\'exico City, C.P. 07000, M\'exico}

\medskip
and {\sl Centro de Investigaci\'on en Matem\'aticas

Callej\'on Jalisco s/n, Valenciana, Guanajuato, Gto, C.P. 36240, 
M\'exico}

\bigskip\medskip

Peter Landweber\quad {\tt landwebe@math.rutgers.edu}

\medskip
{\sl Department of Mathematics, Rutgers University

110 Frelinghuysen Rd,  Piscataway, NJ 08854-8019, USA}

\bigskip\medskip

Thomas Shimkus\quad {\tt shimkust2@scranton.edu}

\medskip
{\sl Department of Mathematics, University of Scranton

Scranton, PA 18510-4666, USA}

\end{document}